\documentclass[12pt]{article}
\usepackage{fullpage}
\usepackage[margin=1.5cm]{geometry}
\usepackage{csquotes}
\usepackage{multirow}
\usepackage{tabularx}
\usepackage{enumitem}
\usepackage{braket}
\usepackage{amsmath}
\usepackage{physics}
\usepackage{titling}
\usepackage{graphicx}
\usepackage{amsmath,amsfonts,amssymb}
\usepackage{physics}
\usepackage{amsthm}
\usepackage{tikz}
\usepackage{pgfplots}
\usepackage{mathrsfs}
\usepackage{musicography}
\usepackage{esint}
\usepackage{hieroglf}
\usepackage{bbm}
\usepackage{amsmath}

\usepackage[T1]{fontenc}
\usepackage[utf8]{inputenc}
\usepackage{tikz}

\usepackage[expansion=false]{microtype}
\usetikzlibrary{arrows.meta}
\usepackage[affil-it]{authblk}

\numberwithin{equation}{section}

\usepackage{aliascnt} 
\usepackage{hyperref} 
\usepackage[capitalise,nameinlink]{cleveref}

\newtheorem{theorem}{Theorem}[section]

\newaliascnt{corollary}{theorem}
\newtheorem{corollary}[corollary]{Corollary}
\aliascntresetthe{corollary}

\newaliascnt{lemma}{theorem}
\newtheorem{lemma}[lemma]{Lemma}
\aliascntresetthe{lemma}

\newaliascnt{proposition}{theorem}
\newtheorem{proposition}[proposition]{Proposition}
\aliascntresetthe{proposition}

\newaliascnt{rhp}{theorem}
\newtheorem{rhp}[rhp]{Riemann--Hilbert Problem}
\aliascntresetthe{rhp}

\theoremstyle{definition}

\newaliascnt{definition}{theorem}
\newtheorem{definition}[definition]{Definition}
\aliascntresetthe{definition}

\newaliascnt{remark}{theorem}
\newtheorem{remark}[remark]{Remark}
\aliascntresetthe{remark}

\crefname{rhp}{RHP}{RHPs}
\Crefname{rhp}{RHP}{RHPs}
\crefname{equation}{}{}
\Crefname{equation}{}{}

\makeatletter
\def\widebreve{\mathpalette\wide@breve}
\def\wide@breve#1#2{\sbox\z@{$#1#2$}%
     \mathop{\vbox{\m@th\ialign{##\crcr
\kern0.08em\brevefill#1{0.8\wd\z@}\crcr\noalign{\nointerlineskip}%
                    $\hss#1#2\hss$\crcr}}}\limits}
\def\brevefill#1#2{$\m@th\sbox\tw@{$#1($}%
  \hss\resizebox{#2}{\wd\tw@}{\rotatebox[origin=c]{90}{\upshape(}}\hss$}
\makeatletter

\makeatletter
\usepackage{pict2e,picture}
\pgfkeys{/csteps/inner ysep/.initial=4pt,
    /csteps/inner xsep/.initial=4pt,
    /csteps/inner color/.initial=red,
    /csteps/outer color/.initial=blue,
}
\newsavebox\csteps@CBox
\newlength\csteps@XLength \newlength\csteps@YLength \newlength\csteps@YDepth \newlength\csteps@tmplen
\def\csteps@CircledParam#1#2{\sbox\csteps@CBox{#2}%
    \csteps@XLength=\wd\csteps@CBox\advance\csteps@XLength by\pgfkeysvalueof{/csteps/inner xsep}\relax
    \csteps@tmplen=\pgfkeysvalueof{/csteps/inner ysep}\relax
    \csteps@YDepth=\dp\csteps@CBox\advance\csteps@YDepth by 0.5\csteps@tmplen\relax
    \csteps@YLength=\ht\csteps@CBox\advance\csteps@YLength by\dp\csteps@CBox\advance\csteps@YLength by\pgfkeysvalueof{/csteps/inner ysep}\relax
    \typeout{DBG:#2\space X\space\the\csteps@XLength\space Y:\the\csteps@YLength\space D:\the\csteps@YDepth}%
    \raisebox{-#1\csteps@YDepth}{%
    \ifdim\csteps@XLength>\csteps@YLength
    \makebox[\csteps@XLength]{
        \makebox(0,\csteps@YLength){%
            \color{\pgfkeysvalueof{/csteps/outer color}}\put(0,0){\oval(\csteps@XLength,\csteps@YLength)}%
        }%
    \makebox(0,\csteps@YLength){%
        \put(-.5\wd\csteps@CBox,0){\textcolor{\pgfkeysvalueof{/csteps/inner color}}{#2}}%
    }}%
    \else
    \makebox[\csteps@YLength]{%
        \makebox(0,\csteps@YLength){%
            \color{\pgfkeysvalueof{/csteps/outer color}}\put(0,0){\circle{\csteps@YLength}}%
        }%
    \makebox(0,\csteps@YLength){%
        \put(-.5\wd\csteps@CBox,0){\textcolor{\pgfkeysvalueof{/csteps/inner color}}{#2}}%
     }}%
    \fi
    }%
}
\def\Circled#1{\csteps@CircledParam{1}{#1}}
\def\CircledTop#1{\csteps@CircledParam{0}{#1}}
\makeatother

\tikzset{/csteps/inner ysep=10pt}
\tikzset{/csteps/inner xsep=10pt}

\begin{document}
\title{A Riemann--Hilbert representation for Sobolev orthogonal polynomials}
\author{Alex Little\thanks{CNRS, Laboratoire de
Physique (LPENS), École Normale Supérieure de Lyon,  46, allée d'Italie, Lyon, France \texttt{alexander.little@ens-lyon.fr}}}
\affil{Laboratoire de Physique \\ ENS de Lyon}
\date{\today}
\maketitle
\abstract{In this article we consider polynomials orthogonal with respect to the inner product
$$\langle P,Q  \rangle_S = \int_{\mathbb{R}}P(x)Q(x) \mathrm{e}^{-V(x)} \, \mathrm{d}x +  \lambda\int_{\mathbb{R}}P^\prime(x)Q^\prime(x) \mathrm{e}^{-V(x)} \, \mathrm{d}x$$
where $V$ is a polynomial of even degree (at least four) and positive leading coefficient, and $\lambda > 0$ a constant. The above inner product is a special case of the so-called Sobolev inner product. We show how one can represent the associated Sobolev orthogonal polynomials as a species of Type I multiple-orthogonal polynomial. This correspondence makes use of the WKB asymptotics of a certain second order ODE. From this we may write a Riemann--Hilbert problem for the Sobolev orthogonal polynomials, from which one can deduce a Christoffel--Darboux-type formula for the projection kernel.}

\tableofcontents

\newpage

\section{Introduction}

The theory of orthogonal polynomials is a well-developed and elegant mathematical theory, with connections with approximation theory, integrable systems and random matrix theory, to name only a few. One way to motivate this theory is to consider the problem of \textit{best polynomial approximation}. Let $f$ be a continuous and bounded function on $\mathbb{R}$. Then we look for the polynomial $P$ of degree at most $N-1$ which minimises
\begin{align}\label{L2norm}
\int_{\mathbb{R}} (f(x) - P(x) )^2 w(x) \, \mathrm{d}x
\end{align}
where $w(x) \geq 0$ is called the \textit{weight function}. For the problem to be well-posed for all $N$ we should require $w$ to have positive total mass and finite moments to all orders. Then, by the Hilbert projection theorem, this problem may be solved in terms of \textit{orthogonal polynomials} with respect to the inner product
\begin{align}\label{orthogpolyip}
\langle P, Q \rangle = \int_{\mathbb{R}} P(x) Q(x) w(x) \, \mathrm{d}x \, .
\end{align}
That is, we search for a sequence of polynomials $\{ \Pi_n \}_{n \in \mathbb{N}}$ such that $\Pi_n$ is monic of degree exactly $n$, and $\langle \Pi_n, \Pi_m \rangle = 0$ for all $n \neq m$. The polynomial $P$ that minimises \eqref{L2norm} is the orthogonal projection, with respect to $\langle \cdot, \cdot \rangle$, of $f$ onto the subspace spanned by $\Pi_0, \dots, \Pi_{N-1}$.

However in many applications one is not interested merely in approximating the pointwise value of the function, but also other aspects of its shape, such as its derivative. Thus one could pose the following problem. Given some function $f$ on $\mathbb{R}$ such that $f$ and $f^\prime$ are both continuous and bounded, we search for a polynomial $P$ of degree at most $N-1$ that minimises
\begin{align}\label{sobL2norm}
\int_{\mathbb{R}} (f(x) - P(x) )^2 w_0(x) \, \mathrm{d}x + \int_{\mathbb{R}} (f^\prime(x) - P^\prime(x) )^2 w_1(x) \, \mathrm{d}x 
\end{align}
where $w_0, w_1$ are a pair of weight functions. In this case, one is led to consider the \textit{Sobolev inner product}
$$\langle P, Q \rangle_S = \int_{\mathbb{R}} P(x) Q(x) w_0(x) \, \mathrm{d}x + \int_{\mathbb{R}} P^\prime(x) Q^\prime(x) w_1(x) \, \mathrm{d}x$$
and we define the sequence of monic Sobolev orthogonal polynomials $\{ P_n \}_{n \in \mathbb{N}}$ to be the sequence of polynomials such that $P_n$ is monic of degree exactly $n$, and $\langle P_n, P_m \rangle_S = 0$ for all $n \neq m$. In the same way as before, the polynomial $P$ of degree at most $N-1$ that minimises \eqref{sobL2norm} is the orthogonal projection, with respect to $\langle \cdot, \cdot \rangle_S$, of $f$ onto the subspace spanned by $P_0, \dots, P_{N-1}$. It is, of course, possible to consider even higher derivatives than simply the first derivative, and replace the weight functions with more general measures, however in this article we shall only be interested in the first derivative case.

Sobolev orthogonal polynomials are harder to study than ordinary orthogonal polynomials because the presence of the derivatives spoils many useful properties; for example, the multiplication operator is no longer symmetric. In response to these obstacles, a wide variety of methods have been developed in the subject, reviewed in \cite{marcellan_xu_2015}. A notable gap in the theory of Sobolev orthogonal polynomials is the lack of any \textit{Riemann--Hilbert representation}. A Riemann--Hilbert representation involves representing a quantity of interest in terms of the solution to a matrix-valued boundary value problem in the complex plane, and can be regarded as a non-commutative version of a contour integral \cite{its_2003}. It was first observed by Fokas, Its and Kitaev that orthogonal polynomials have a Riemann--Hilbert representation \cite{fokas_its_kitaev_1992}, and with the development of the nonlinear steepest descent method for Riemann--Hilbert problems, it became possible to extend Plancherel--Rotach asymptotics to non-classical weight functions (see the textbook of Deift \cite{deift_1999}). This was a theoretical triumph that yielded, as a corollary, a proof of local universality of eigenvalues in a large class of random matrix models. There also exist Riemann--Hilbert representations for multiple orthogonal polynomials \cite{vanassche_etal_2001}, bi-orthogonal polynomials \cite{KUIJLAARS2005313}, orthogonal polynomials on the unit circle \cite{11ff0b99-fce0-3547-a87b-f88f323344b3,martinezfinkelshtein2006szegopolynomialsviewriemannhilbert} and more recently skew-orthogonal polynomials \cite{PIERCE2008230,little_skew_op}. For some of these cases a successful nonlinear steepest descent analysis has been carried out, though the more singular the weight function and the higher dimensional the RHP, the more difficult the analysis.

In the present article we address this gap in the literature by presenting a Riemann--Hilbert representation for a certain class of Sobolev orthogonal polynomials. This is done by returning to a method employed in the earliest studies of Sobolev orthogonal polynomials, namely the method of \textit{integration by parts}. This method was used in the studies by Althammer \cite{GDZPPN002179970} and Brenner \cite{Brenner1972} who considered, respectively, the cases $(w_0, w_1) = (\chi_{[-1,1]}, \lambda \chi_{[-1,1]})$ and  $(w_0, w_1) = (\mathrm{e}^{-x} \chi_{\mathbb{R}_+}, \mathrm{e}^{-x}\chi_{\mathbb{R}_+})$. We develop this method in a way that allows us to study a particular class of Sobolev orthogonal polynomials: those for which $w_0(x) = \mathrm{e}^{-V(x)}$ and $w_1(x) =\lambda \mathrm{e}^{-V(x)}$, where $\lambda > 0$ and $V$ is a polynomial of even degree and positive leading coefficient. The central achievement of this article is to show that such Sobolev orthogonal polynomials can be regarded as a particular kind of \textit{multiple orthogonal polynomial}. With this in place, the standard theory of multiple-orthogonal polynomials \cite{vanassche_etal_2001} allows one to write a Riemann--Hilbert formulation (\cref{RHP1}) for these polynomials. From this we are able to derive a Christoffel--Darboux-type formula (\cref{CDtheorem}).

Let us be more precise. Let $\lambda > 0$ and $V$ a polynomial such that
\begin{align}
V(x) = \gamma x^D + \mathcal{O}(x^{D-1}) & & \text{ for } D\geq 4 \text{ even, and } \gamma > 0 \, .
\end{align}
We define $P_n$ for $n \in \mathbb{N}$ to be the monic polynomial of degree exactly $n$ such that
\begin{align*}
\int_{\mathbb{R}} x^k P_n(x) \mathrm{e}^{-V(x)}\, \mathrm{d}x + \lambda \int_{\mathbb{R}}  \frac{\mathrm{d}}{\mathrm{d}x} (x^k) P_n^\prime(x) \mathrm{e}^{-V(x)}\, \mathrm{d}x = 0 \, , & & \forall k \in [\![ 0, n-1]\!] \, .
\end{align*}
If we integrate by parts in the second integral, then we have
\begin{align}\label{Psiorthog2}
\int_{\mathbb{R}} x^k \Psi_n(x) \mathrm{e}^{-V(x)}\, \mathrm{d}x = 0 \, , & & \forall k \in [\![ 0, n-1]\!]
\end{align}
where
\begin{align}\label{PsiPintro}
\Psi_n(x) \overset{\mathrm{def}}{=} \nu_n \Big( V^\prime(x) P_n^\prime(x) - P_n^{\prime\prime}(x) + \lambda^{-1}P_n(x) \Big)  & & \text{ where } \, \nu_n = \begin{cases} \frac{1}{n D \gamma} & n \geq 1 \\
\lambda & n = 0
\end{cases}
\end{align}
is the \textit{dual polynomial}. It has been rescaled so that $\Psi_n$ is monic, where, for $n \geq 1$,  $\Psi_n$ has degree exactly $n+D-2$, since for $D \geq 4$, $V^\prime P^\prime_n$ is the leading term in \eqref{PsiPintro}. If $n = 0$, $\Psi_0 \equiv 1$ which has degree $0$. On the other hand, if $D =2 $, this is no longer true since $V^\prime P^\prime_n$ and $\lambda^{-1}P_n$ have the same degree; and indeed  the $D=2$ case is genuinely anomalous and requires a different treatment compared with $D \geq 4$. Indeed, when $D = 2$, $\Psi_n$ has degree at most $n$ and hence by \eqref{Psiorthog2}, $\Psi_n$ is proportional to the $n$th orthogonal polynomial with respect to the weight function $\mathrm{e}^{-V(x)}$. From the knowledge of $\Psi_n$ one could then reconstruct $P_n$. However when $D \geq 4$, the orthogonality conditions \eqref{Psiorthog2} are not enough to determine $\Psi_n$, and one must find additional constraints. It is precisely this problem that the present article solves. We remark that the method of solution is strongly inspired by the author's previous article on skew-orthogonal polynomials \cite{little_skew_op}, in which a similar underdetermination problem arises.

If we now return to the problem of minimising \eqref{sobL2norm}, the solution is given by
\begin{align*}
P = \sum_{n=0}^{N-1} \frac{\langle f, P_n \rangle_S}{h_n} \,  P_n \, , & &\text{ where } \,  h_n \overset{\mathrm{def}}{=} \langle P_n, P_n \rangle_S \, .
\end{align*}
If one integrates by parts, one may write the solution as
\begin{align}\label{kernelformula}
P(x) = \int_{\mathbb{R}} \mathbb{K}_N(x,y) f(y) \mathrm{e}^{-V(y)}\, \mathrm{d}y \, , & \text{ where } \mathbb{K}_N(x,y) \overset{\mathrm{def}}{=} \sum_{n=0}^{N-1} \frac{P_n(x)}{h_n} \Big[ P_n(y) - \lambda P_n^{\prime \prime}(y) + \lambda  V^\prime(y) P_n^\prime(y)\Big] \, .
\end{align}
There are a variety of kernels one could construct out of the Sobolev orthogonal polynomials, but we argue that $\mathbb{K}_N$, despite not being symmetric, is the most natural since it solves the associated polynomial approximation problem. With these definitions in place, let us now state our main results.
\begin{enumerate}[label=(\arabic*)]
\item We show that it is possible to represent the dual Sobolev polynomials $\{ \Psi_n \}_{n \geq 1}$ as a particular kind of Type II multiple orthogonal polynomials (\cref{dualsobolevMOP}). From this we are able to represent the Sobolev orthogonal polynomials $\{ P_n \}_{n \geq 1}$ as a particular kind of Type I multiple orthogonal polynomials (\cref{sobolevMOP}).
\item From this, one is able to write a pair of $D \times D$ matrix Riemann--Hilbert problems $A_n$ and $\widehat{A_n} = A_n^{-\mathsf{T}}$ (see RHPs \ref{RHP1} and \ref{RHP2}) such that $\Psi_n = (A_n)_{11}$ and $P_n = (\widehat{A_n})_{22} = (A_n^{-1})_{22}$.
\item Finally, we deduce a Christoffel--Darboux-type formula (\cref{CDtheorem}). More precisely, we have that
\begin{align}\label{CDformula}
\mathbb{K}_N(x,y) = - \frac{1}{2 \pi \mathrm{i}} \frac{(A_N^{-1}(x) A_N(y))_{21}}{x-y} \, .
\end{align}
We call this a \enquote{Christoffel--Darboux-type formula} since if one considers the $2 \times 2$ Fokas--Its--Kitaev Riemann--Hilbert problem for orthogonal polynomials, the right hand side of \eqref{CDformula} gives the usual Christoffel--Darboux formula. We remark that since the kernel \eqref{CDformula} takes the form of something finite rank divided by the difference of the two arguments, it is the kernel of an Its--Izergin--Korepin--Slavnov integrable operator \cite{deift_integrable_operators}.
\end{enumerate}
Thus, if one could do a nonlinear steepest descent analysis of $A_N$ as $N \to +\infty$, one would obtain the asymptotics of $\Psi_N$, $P_N$ and the projection kernel $\mathbb{K}_N$. There is, however, a difficulty in addition to all the usual difficulties of analysing higher-dimensional Riemann--Hilbert problems. This is that the elements in the jump matrix are not explicit, but are rather certain special solutions $Y_k$ to a certain second order ODE, and their pairwise Wronskians. One cannot expect in general to have explicit formulas for these quantities, since it amounts to solving a certain spectral problem, although one can write series solutions. However, we would defend the usefulness of our RHP in two ways. Firstly, that the spectral data depend only on $V$ and $\lambda$, not on $N$, so they are given once and for all if $V$ and $\lambda$ are fixed. Secondly, by tools such as exact WKB analysis, it may be possible to obtain asymptotic approximations of the spectral data in the regimes where $\lambda \to 0$ or $\lambda \to +\infty$, and also if we rescale the potential $V = N v$. We discuss the usefulness of our RHP in more detail in \cref{sec:discussion}.

To place our results in context, let us highlight certain past asymptotic results relevant to the regime we are considering. The $n \to +\infty$ asymptotics of the ratio $P_n(x)/ \Pi_n(x)$ in the quartic regime $V(x) = x^4$, and where $x$ is bounded away from the real axis,  has been studied in the literature by making use of relationships between Sobolev orthogonal polynomials and the corresponding ordinary orthogonal polynomials. This has been done both in the regime of $\lambda$ fixed \cite{CACHAFEIRO200326} and in the so-called \enquote{balanced scaling} regime $\lambda_n \to 0$ \cite{alfaro_etal_2009}.  It has also been shown that in the regime of $\lambda$ fixed, $V$ convex, $P_n^\prime(x)$ asymptotically behaves like $n \Pi_{n-1}(x)$ \cite{geronimo_lubinsky_marcellan_2005}. From the known asymptotics of $\Pi_{n-1}$ one may then deduce the asymptotics of $P_n$. These methods, however, are limited to certain regimes, and are unable to access the scaling behaviour of the kernel $\mathbb{K}_N$, the local \enquote{edge} behaviour of $P_n$, or the local oscillations of $P_n$ in the regime of \enquote{balanced scaling.}

Finally, we remark that much of the recent work on Sobolev orthogonal polynomials has centred on the theory of \textit{coherent pairs} (see Section 5 of \cite{marcellan_xu_2015}). It was shown by Meijer \cite{MEIJER1997321} that for $(w_0, w_1)$ to be coherent at least one of these weight functions must be classical and the other an explicitly determined perturbation of it. It follows that our case $(w_0, w_1) = (\mathrm{e}^{-V}, \lambda\mathrm{e}^{-V})$ can never be coherent for $D\geq 4$. Our results therefore are interesting because they fall outside of the \enquote{coherent} class.

\subsection{Notation}

\begin{enumerate}[label=(\arabic*)]
\item We denote the integer interval $[\![ a, b ]\!] = \mathbb{Z} \cap [a,b]$, for $a,b \in \mathbb{R}$.
\item Given a set $A$, we denote $\chi_A$ to be the indicator function of $A$, i.e. $\chi_A(x) = 1$ if $x \in A$ and $\chi_A(x) = 0$ if $x \not\in A$.
\item We let $\mathcal{M}_D(\mathbb{C})$ denote the space of $D \times D$ matrices with elements in $\mathbb{C}$.
\item We let $\mathscr{P}_k$ stand for the space of real polynomials on $\mathbb{R}$ of degree at most $k$, and let $\mathscr{P} = \cup_{k \geq 0}\mathscr{P}_k$ be the space of all polynomials. We let $\mathscr{P}_k^{\mathbb{C}}$ be space of complex polynomials of degree at most $k$.
\item We let $X^k$ stand for the monomial function $X^k(x) =x^k$.
\item Let $\Gamma \subset \mathbb{C}$ be a piecewise smooth contour in $\mathbb{C}$. Then the \textit{Cauchy transform} of the measurable function $f\colon \Gamma \longrightarrow \mathbb{C}$ is defined as
\[
C_\Gamma (f) (z) = \frac{1}{2\pi {\rm i}}\int_\Gamma \frac{f(x)}{x-z} {\rm d}x,\qquad z \in \mathbb{C}\setminus \Gamma,
\]
whenever this integral exists. In our case, we will always consider sufficiently \enquote{nice} functions $f$ and contours $\Gamma$ such that the Cauchy transform is well-defined and analytic for all $z \in \mathbb{C}\setminus \Gamma$. Furthermore, we shall always consider functions with finite moments~${\int_\Gamma |x|^k |f(x)| |{\rm d}x| < +\infty}$ for all $k\geq 0$ and for which $C_\Gamma (f)(z)$ admits an asymptotic expansion in $\frac{1}{z}$ as $z \to \infty$ in the sense of Poincar\'e. More precisely, for every $K \in \mathbb{N}$, we~have
\begin{align*}
C_\Gamma (f)(z) = - \frac{1}{2\pi {\rm i}} \sum_{k=0}^K z^{-k-1} \int_\Gamma x^k f(x) {\rm d}x +\frac{z^{-K-1}}{2\pi {\rm i} } \int_\Gamma \frac{ x^{K+1} f(x)}{x-z} {\rm d}x \, .
\end{align*}
If $\mathrm{dist}(z,\Gamma) \geq \epsilon > 0$, we find that $\int_\Gamma \frac{ x^{K+1} f(x)}{x-z} {\rm d}x$ is bounded. However, if $f$ is analytic we may deform the contour to achieve boundedness in a full neighbourhood of (complex) infinity. This will be the case in all the Cauchy integrals that appear in this paper. The general theory of such Cauchy integrals for functions of varying degrees of decay and regularity is discussed in notes of Deift \cite{deift2019riemannhilbert} and the book of Muskhelishvili \cite{singular}.
\end{enumerate}

\section{Riemann--Hilbert representation}

For the rest of this article, unless otherwise stated, $V(x) = \gamma x^D + \mathcal{O}(x^{D-1})$ ($x \to \infty$) is a real polynomial of positive leading coefficient $\gamma > 0$ and even degree $D\geq 4$.
\begin{definition} We define $\langle \cdot, \cdot \rangle_S : \mathscr{P}\times \mathscr{P} \longrightarrow \mathbb{R}$ to be the inner product on polynomials defined by
\begin{align}\label{sobolevinnerprod}
\langle P, Q \rangle_S = \int_{\mathbb{R}}P(x) Q(x) \mathrm{e}^{-V(x)}\, \mathrm{d}x + \lambda \int_{\mathbb{R}}P^\prime(x) Q^\prime(x) \mathrm{e}^{-V(x)}\, \mathrm{d}x \, .
\end{align}
We define $\{ P_n \}_{n \in \mathbb{N}}$ to be the sequence of polynomials such that $P_n$ is monic of degree $n$ and $\langle P_n , P_m \rangle_S = 0$ for all $n \neq m$. We call these $P_n$ the $n$th monic Sobolev orthogonal polynomial.
\end{definition}
$\langle \cdot , \cdot \rangle_S $ is clearly symmetric, and from the inequality
$$\langle P, P \rangle_S \geq \int_{\mathbb{R}} P(x)^2 \mathrm{e}^{-V(x)} \, \mathrm{d}x \geq 0$$
we see that if $\langle P, P \rangle_S = 0$, we must have $P \equiv 0$, and so $\langle \cdot , \cdot \rangle_S $ is positive definite. Hence by the Gram--Schmidt procedure we may orthogonalise the sequence $X^0, X^1 , X^2, \dots $ and so arrive at a sequence of monic polynomials orthogonal with respect to  $\langle \cdot , \cdot \rangle_S $. This establishes existence of the sequence of Sobolev orthogonal polynomials, and by positive definiteness this sequence is unique.

Motivated by this, let us define the \textit{dual polynomial}
\begin{align}\label{PsiP}
\Psi_n(z) \overset{\mathrm{def}}{=} \nu_n \Big( V^\prime(z) P_n^\prime(z) - P_n^{\prime\prime}(z) + \lambda^{-1}P_n(z) \Big)  & & \text{ where } \, \nu_n = \begin{cases} \frac{1}{n D \gamma} & n \geq 1 \\
\lambda & n = 0\, .
\end{cases}
\end{align}
By construction, $\Psi_n$ is monic of degree $n + D-2$ for $n \geq 1$ and $\Psi_0 \equiv 1$. From the orthogonality condition for $P_n$, we have $\langle X^k , P_n \rangle_S = 0$ for all $k \in [\![ 0, n-1 ]\!]$. Integrating by parts we find that $\Psi_n$ satisfies
\begin{align}\label{Psiorthog}
\int_{\mathbb{R}} x^k \Psi_n(x) \mathrm{e}^{-V(x)} \, \mathrm{d}x = 0 \, , & & \forall k \in [\![ 0, n-1 ]\!] \, .
\end{align}
These orthogonality conditions, however, leave $\Psi_n$ underdetermined since it supplies $n$ conditions but one has $n+D-2$ degrees of freedom in total, leaving $D-2$ degrees of freedom remaining. The remaining degrees of freedom are fixed by the fact that $\Psi_n$ lies in the image of the map $P \mapsto P - \lambda P^{\prime\prime} + \lambda V^\prime P^\prime$.
\begin{remark} In the case $D = 2$, where without loss of generality we may take $V(x) = x^2$, the polynomial $V^\prime P_n^\prime - P_n^{\prime\prime} + \lambda^{-1}P_n$ has degree at most $n$ and is orthogonal to $X^0, \dots, X^{n-1}$ with respect to \eqref{orthogpolyip} for $w(x) = \mathrm{e}^{-x^2}$. It must therefore be proportional to the $n$th monic Hermite polynomial $H_n$, orthogonal with respect to \eqref{orthogpolyip} for the aforementioned choice of $w$. By comparing the coefficient of $x^n$ we see that
$$2x P_n^\prime(x) - P_n^{\prime\prime}(x) + \lambda^{-1}P_n(x) = (2 n  + \lambda^{-1})H_n(x) \, .$$
From the differential equation satisfied by $H_n$, one immediately sees that $P_n = H_n$ is the unique polynomial solution to the above ODE. This illustrates the well-known fact that the case $D=2$ for \eqref{sobolevinnerprod} is trivial (see Section 3 of \cite{alfaro_etal_2009}), in the sense that it reduces to classical Hermite polynomials. For this reason, in this article we shall only be interested in the case $D \geq 4$.
\end{remark}
\begin{lemma}
If $D > 2$, then the map $\mathscr{P}_n \to \mathscr{P}_{n+D-2} $, $P \mapsto P - \lambda P^{\prime\prime} + \lambda V^\prime P^\prime$ is injective. 
\end{lemma}
\begin{proof}
To show this we need only show that the kernel is trivial, hence suppose that $ P(z) - \lambda P^{\prime\prime}(z) + \lambda V^\prime(z) P^\prime(z) = 0$. If $P \neq 0$, then there are two possibilities: either $P$ is constant or $P$ is non-constant. If $P$ is constant then the left hand side is equal to $P \neq 0$ which is a contradiction. If $P$ is non-constant, then the left hand side has degree exactly $\deg P + D-2 > 0$, which is also a contradiction, hence $P = 0$.
\end{proof}
From this injectivity, it follows that the image of this map $\mathscr{P}_n \to \mathscr{P}_{n+D-2} $ is a codimension $D-2$ subspace. To characterise this subspace we should find $D-2$ independent linear functionals which vanish on this space, and in this way supply the \enquote{missing} $D-2$ orthogonality conditions. It is of course clear that, abstractly, such linear functionals must exist. However such abstract arguments give us no formula, and \textit{a priori} such linear functionals depend on the dimension $n$. We thus search for a more explicit way to characterise the image of the map.

\subsection{The homogeneous problem and Sibuya solutions}

The key idea is to reverse our perspective and regard \eqref{PsiP} as a differential equation for $P_n$ in which $\Psi_n$ plays the role of a source term. Then we should ask: given a polynomial source term $\Psi_n$, when does there exist a polynomial solution $P_n$? We may solve \eqref{PsiP} by the method of variation of parameters. Hence we should begin by considering the homogeneous version of the ODE. Thus let us consider 
\begin{align}\label{Fequation}
 V^\prime(z) F^\prime(z) - F^{\prime\prime}(z) + \lambda^{-1}F(z) = 0 \, .
\end{align}
Once we have fully understood the space of solutions of \eqref{Fequation}, we can analyse the solutions of \eqref{PsiP}.

We begin by remarking that \eqref{Fequation} may be transformed into a Schrödinger equation. If one lets $\psi(z) = F(z) \mathrm{e}^{-\frac{1}{2}V(z)}$ one finds that
\begin{align}\label{psiequation}
\psi^{\prime \prime}(z) = U(z) \psi(z) & &  \text{ where  } \,  U(z) = \frac{1}{4}V^\prime(z)^2 - \frac{1}{2}V^{\prime\prime}(z) + \lambda^{-1} \, .
\end{align}
For reasons that shall become clear later, we are interested in the behaviour of the solutions of \eqref{psiequation} as $z \to \infty$ in different directions in $\mathbb{C}$. Intuitively, when $z$ becomes large the \enquote{potential} $U(z)$ also becomes large. Hence one expects the $z \to \infty$ behaviour to be described by Liouville--Green (WKB) asymptotics, and this is indeed the case.

Naively following WKB theory, if we let $\phi(z) = \int^z \sqrt{U(w)}\, \mathrm{d}w $ be the \enquote{phase} function then
\begin{align*}
\phi(z) = \frac{1}{2}V(z) - \frac{1}{2}\log V^\prime(z) + \mathcal{O}(1) \, , & & z \to \infty \, .
\end{align*}
hence the Liouville--Green asymptotics gives
\begin{align*}
\frac{1}{\sqrt{V^\prime(z)}} \mathrm{e}^{ \pm \frac{1}{2} ( V(z) - \log V^\prime(z)) + \mathcal{O}(1)} & & z\to \infty \, .
\end{align*}
If we take the $+$ sign we get the asymptotics $V^\prime(z)^{-1}\mathrm{e}^{\frac{1}{2}V(z)}$ and, if we take the $-$ sign, $\mathrm{e}^{-\frac{1}{2}V(z)}$. Let
\begin{align}\label{infinity}
\infty_k \overset{\mathrm{def}}{=} \mathrm{e}^{\mathrm{i}\pi k /D} \infty \, , & & k \in [\![ 0, 2D-1 ]\!]
\end{align}
then for $k$ even the $\mathrm{e}^{-V(z)}$ solution decays as $z \to \infty_k$ and the $V^\prime(z)^{-1}\mathrm{e}^{\frac{1}{2}V(z)}$ solution blows up, whilst for $k$ odd it is the other way round. Given that \eqref{psiequation} has two linearly independent solutions, we expect that as $z \to \infty_k$ there is, up to constant multiplication, only one solution of \eqref{psiequation} which decays, while the other solution blows up. We should emphasise that such asymptotics only hold in a sector of $\mathbb{C}$, so that the decaying solution when continued to another sector will no longer decay. This is the Stokes phenomenon. The rigorous construction of such solutions and their relations is the subject of the celebrated monograph of Sibuya \cite{sibuya_1975}. For our purposes, because our potential $U$ has a very special structure, we can do a self-contained construction of these Sibuya solutions.

\begin{remark} We note that if we define $U_0(z) = \frac{1}{4}V^\prime(z)^2 - \frac{1}{2} V^{\prime\prime}(z)$, so that $U(z) = U_0(z) + \lambda^{-1}$, then $U_0$ takes the form of a \enquote{supersymmetric} potential studied by Witten in \cite{witten_1982} and $-\lambda^{-1}$ behaves like a spectral parameter. In particular, it was observed that one has the exact factorisation 
\begin{align}\label{factorisation}
 - \frac{\mathrm{d}^2}{\mathrm{d}x^2} + U_0(x) = \Big( - \frac{\mathrm{d}}{\mathrm{d}x} + \frac{1}{2} V^\prime(x) \Big)\Big(\frac{\mathrm{d}}{\mathrm{d}x} + \frac{1}{2} V^\prime(x) \Big) \, .
\end{align}
Whether this connection with supersymmetry has any deep implications for the theory of Sobolev orthogonal polynomials is unclear. However the factorisation \eqref{factorisation} will become useful later because it implies the above operator is positive semidefinite on $L^2(\mathbb{R})$.
\end{remark}

In order to construct the Sibuya solutions to \eqref{psiequation}, it is actually better to return to \eqref{Fequation}. It is convenient to introduce the following change of variables on a neighbourhood of $\infty$.
\begin{definition}\leavevmode
\begin{enumerate}[label=(\arabic*)]
\item Consider the function
\begin{align*}
\varphi(z) \overset{\mathrm{def}}{=} ( \gamma^{-1} V(z))^{1/D} \, ,
\end{align*}
defined on some neighbourhood of infinity. The choice of branch is made by requiring that $\varphi(z) = z ( 1+ \mathcal{O}(z^{-1}))$ as $z \to \infty$. Define $\varrho > 0$ to be a constant sufficiently large so that $\varphi$ and $\varphi^{-1}$ are both injective on $\mathcal{U}_{\varrho} := \{ z \in \mathbb{C} \, : \, |z| > \frac{1}{2} \varrho \}$. Note that, by construction, $V(z) = \gamma \varphi(z)^D$. Furthermore, by making $\varrho$ sufficiently large we can make it so that
\begin{align}\label{phizratio}
\frac{1}{2}|z| \leq |\varphi(z) | \leq 2 |z| & & \forall |z| \geq \varrho \, .
\end{align}
Let us also suppose that $\varrho$ is sufficiently large so that $\frac{(D-1)\varrho^{-D}}{\gamma D} \leq \frac{1}{2}$, an assumption which will be useful later.
\item Let us then define
\begin{align*}
\Omega_k \overset{\mathrm{def}}{=} \varphi^{-1}\Big( \Big\{ w \in \mathbb{C} \, : \, |w| \geq \varrho \, \text{ and } \,  \pi \frac{k-1}{D} \leq  \arg w \leq \pi \frac{k+1}{D} \Big\}  \Big) \, , & & k \in [\![0, 2D-1 ]\!] \, .
\end{align*}
\item Then for $z \in \Omega_k$, define the contour
\begin{align*}
\tau_{k,z}(s) = \begin{cases} \varphi^{-1} \big( \,  |\varphi(z)| \,  \mathrm{e}^{\mathrm{i}(1-s) \arg \varphi(z) + \mathrm{i}\pi k s/D}\big) & s \in [0,1] \\
\varphi^{-1} \big(|\varphi(z)| \, \mathrm{e}^{\mathrm{i}\pi k/D} s \big)  & s \in [1,+\infty) \, .
\end{cases}
\end{align*}
$\Gamma_{k,z} \overset{\mathrm{def}}{=} \tau_{k,z}([0,+\infty)) \subset \mathbb{C}$
\end{enumerate}
\end{definition}
$\varphi$ is essentially a conformal deformation on a neighbourhood of $\infty$, asymptotically equal to the identity, which makes $V$ a monomial in the new variable. The path $\tau_{k,z}$ for the case $D=4$, $k=1$, is depicted in Figure \ref{fig:taukz}.

\begin{figure}[tbp]
\centering
\def\rr{1.8}       
\def\rz{3.0}       
\def\Rout{4.7}     
\def\thz{15}       
\def\thk{45}       
\def\thm{0}        
\def\thp{90}       
\def\thmid{30}     
\def\bt{1.1}
\def\ps{135}
\def\Px#1#2{(#2)*cos(#1)+(\bt/(#2))*cos(\ps-(#1))}
\def\Py#1#2{(#2)*sin(#1)+(\bt/(#2))*sin(\ps-(#1))}
\begin{tikzpicture}[
    scale=0.85,
    every node/.style={font=\small},
    tip/.style={-{Stealth[length=2.6mm,width=2mm]}},
    pathline/.style={very thick},
    bdry/.style={densely dashed,gray!65,thin},
    ax/.style={-{Stealth[length=1.8mm]},gray!40,thin},
    sect/.style={fill=blue!7},
    disk/.style={fill=black!12},
    cp/.style={variable=\t,smooth,samples=61}   
  ]

\begin{scope}
  \fill[sect] (\thm:\rr) arc (\thm:\thp:\rr) -- (\thp:\Rout)
              arc (\thp:\thm:\Rout) -- cycle;
  \fill[disk] (0,0) circle (\rr);
  \draw[gray!55,thin] (0,0) circle (\rr);
  \draw[ax] (-2.7,0) -- ({\Rout+0.7},0);
  \draw[ax] (0,-2.7) -- (0,{\Rout+0.7});
  \draw[bdry] (0,0) circle (\rz);
  \draw[bdry] (\thm:\rr) -- (\thm:{\Rout+0.35});
  \draw[bdry] (\thp:\rr) -- (\thp:{\Rout+0.35});
  \draw[bdry] (0,0) -- (\thk:\rz);
  \draw[gray!70,thin] (\thm:0.85) arc (\thm:\thk:0.85);
  \draw[pathline,tip] (\thz:\rz) arc (\thz:\thmid:\rz);
  \draw[pathline]     (\thmid:\rz) arc (\thmid:\thk:\rz);
  \draw[pathline,tip] (\thk:\rz) -- (\thk:{\Rout+0.8});
  \fill (\thz:\rz) circle (1.8pt);
  \fill (\thk:\rz) circle (1.2pt);
  \node[anchor=north west,inner sep=1.5pt] at (\thz:\rz) {$\varphi(z)$};
  \node[anchor=south west] at (\thk:{\Rout+0.8}) {$\infty_1$};
  \node[anchor=west,inner sep=1pt] at (21:0.95) {\scriptsize$\tfrac{\pi}{4}$};
  \node[anchor=east,gray!65,fill=white,inner sep=1.5pt]
        at (210:{\rr-0.05}) {$|w|=\varrho$};
  \node[anchor=west,gray!65,fill=white,inner sep=1.5pt]
        at (300:\rz) {$|w|=|\varphi(z)|$};
  \node[anchor=north east,inner sep=1pt] at (0,0) {\scriptsize$0$};
  \node at (1.2,-4.2) {(a) the $\varphi$-plane};
\end{scope}

\begin{scope}[shift={(9.6,0)}]
  \fill[sect]
      plot[cp,domain=\thm:\thp] ({\Px{\t}{\rr}},{\Py{\t}{\rr}})
   -- plot[cp,domain=\rr:\Rout] ({\Px{\thp}{\t}},{\Py{\thp}{\t}})
   -- plot[cp,domain=\thp:\thm] ({\Px{\t}{\Rout}},{\Py{\t}{\Rout}})
   -- plot[cp,domain=\Rout:\rr] ({\Px{\thm}{\t}},{\Py{\thm}{\t}})
   -- cycle;
  \fill[disk] plot[cp,domain=0:360,samples=181]
        ({\Px{\t}{\rr}},{\Py{\t}{\rr}});
  \draw[gray!55,thin] plot[cp,domain=0:360,samples=181]
        ({\Px{\t}{\rr}},{\Py{\t}{\rr}});
  \draw[ax] (-2.7,0) -- ({\Rout+0.7},0);
  \draw[ax] (0,-2.7) -- (0,{\Rout+0.7});
  \draw[bdry] plot[cp,domain=0:360,samples=181]
        ({\Px{\t}{\rz}},{\Py{\t}{\rz}});
  \draw[bdry] plot[cp,domain=\rr:{\Rout+0.35}]
        ({\Px{\thm}{\t}},{\Py{\thm}{\t}});
  \draw[bdry] plot[cp,domain=\rr:{\Rout+0.35}]
        ({\Px{\thp}{\t}},{\Py{\thp}{\t}});
  \draw[dotted,gray!80] (0,0) -- (\thk:{\Rout+1.0});
  \draw[pathline,tip] plot[cp,domain=\thz:\thmid]
        ({\Px{\t}{\rz}},{\Py{\t}{\rz}});
  \draw[pathline]     plot[cp,domain=\thmid:\thk]
        ({\Px{\t}{\rz}},{\Py{\t}{\rz}});
  \draw[pathline,tip] plot[cp,domain=\rz:{\Rout+0.8}]
        ({\Px{\thk}{\t}},{\Py{\thk}{\t}});
  \fill ({\Px{\thz}{\rz}},{\Py{\thz}{\rz}}) circle (1.8pt);
  \fill ({\Px{\thk}{\rz}},{\Py{\thk}{\rz}}) circle (1.2pt);
  \node[anchor=north west,inner sep=1.5pt]
        at ({\Px{\thz}{\rz}},{\Py{\thz}{\rz}}) {$z$};
  \node[anchor=south west]
        at ({\Px{\thk}{\Rout+0.8}},{\Py{\thk}{\Rout+0.8}}) {$\infty_1$};
  \node[anchor=south west,inner sep=1pt]
        at ({\Px{28}{\rz+0.22}},{\Py{28}{\rz+0.22}}) {$\tau_{1,z}$};
  \node at ({\Px{62}{4.0}},{\Py{62}{4.0}}) {$\Omega_1$};
  \node[anchor=north east,inner sep=1pt] at (0,0) {\scriptsize$0$};
  \node at (1.2,-4.2) {(b) the $z$-plane};
\end{scope}


\draw[-{Stealth[length=2.6mm]},thick] (6.1,4.7) -- (7.7,4.7)
      node[midway,above] {$\varphi^{-1}$};

\end{tikzpicture}
\caption{The path $\tau_{k,z}$ depicted for the case $D=4$ and $k=1$.}
\label{fig:taukz}
\end{figure}
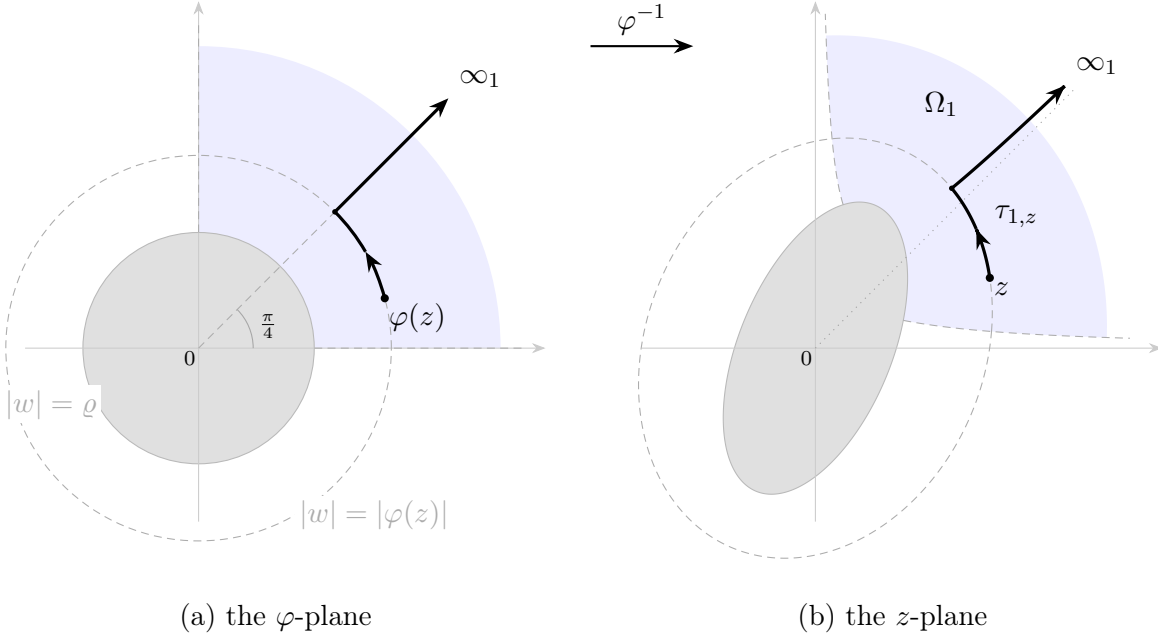
\begin{lemma}\label{progressivity} $\Re V(\tau_{k,z}(s))$ is nondecreasing in $s \in [0,+\infty)$ for $k$ even and nonincreasing for $k$ odd.
\end{lemma}
\begin{proof}
By explicit calculation, one has
\begin{align*}
\Re V(\tau_{k,z}(s)) = (-1)^k  |V(z)|  \begin{cases} \cos\Big((1-s) D \Big\{ \arg \varphi(z) - \pi k/D\Big\} \Big)  & s \in [0,1] \\
 s^D & s \in [1, +\infty) \, .
\end{cases}
\end{align*}
Since, $z \in \Omega_k$, $D \Big\{ \arg \varphi(z) - \pi k/D\Big\} \in [-\pi,\pi]$, $s\mapsto \cos\Big((1-s) D \Big\{ \arg \varphi(z) - \pi k/D\Big\} \Big)$ is nondecreasing for $s \in [0,1]$. Likewise $ s^D$ is clearly nondecreasing for $s \in [1,+\infty)$.
\end{proof}
Moving forward, we make the following claim.
\begin{proposition}\label{FkODEprop} There exists a unique entire function $F_k$ such that
\begin{align}\label{FkODE}
(-1)^k V^\prime(z) F_k^\prime(z) - F_k^{\prime\prime}(z) + \lambda^{-1}F_k(z) = 0
\end{align}
and $F_k(z) = 1 + \mathcal{O}(|z|^{2-D})$ for $z \in \Omega_k$ as $|z| \to +\infty$.
\end{proposition}
For $z \in \Omega_k$ and $w \in \Gamma_{k,z}$, define
$$K_{(-1)^k V}(z,w) = \int_z^w \mathrm{e}^{(-1)^k  \{ V(t)-  V(w) \} } \, \mathrm{d}t$$
where the contour runs along $\Gamma_{k,z}$.
\begin{lemma}\label{KVestimate} For $z \in \Omega_k$ and $w \in \Gamma_{k,z}$, there is a constant $C > 0$ independent of $z$ and $w$ such that
\begin{align*}
|K_{(-1)^k V}(z,w)| \leq C \frac{1}{|w|^{D-1}} & & \forall z \in \Omega_k \text{ and } w \in \Gamma_{k,z} \, .
\end{align*}
\end{lemma}
\begin{proof} Let us take $k$ to be even, since the proof will be identical for $k$ odd. Let us take the contour from $z$ to $w$ in the formula for $K_V(z,w)$ as being the same as $\Gamma_{k,z}$ except terminating at $w$ rather than extending to $\infty_k$. The contour consists of two parts, where $|\varphi(t)| = |\varphi(z)|$ and $|\varphi(t)| > |\varphi(z)|$.

Let us begin by considering the case when $|\varphi(w)| = |\varphi(z)|$. Writing $\mathrm{e}^{V(t)} = \frac{1}{V^\prime(t)} \frac{\mathrm{d}}{\mathrm{d}t}(\mathrm{e}^{V(t)})$ and integrating by parts we find
\begin{align*}
K_V(z,w) = \frac{1}{V^\prime(w)} - \frac{1}{V^\prime(z)} \mathrm{e}^{V(z) - V(w)} + \int_{z}^w \frac{V^{\prime\prime}(t)}{V^\prime(t)^2} \mathrm{e}^{V(t) - V(w)}\, \mathrm{d}t \, .
\end{align*}
By \cref{progressivity}, we have constructed our contour such that $|\mathrm{e}^{V(t) - V(w)}| \leq 1$ and $|\mathrm{e}^{V(z) - V(w)}| \leq 1$. Since $|\varphi(t)| = |\varphi(z)| = |\varphi(w)|$, by \eqref{phizratio} it follows that $|z|,|w|,|t|$ all differ from one another by no more than a factor of $4$. $|V^\prime(w)|^{-1} \leq C |w|^{-D+1}$, and likewise $|V^\prime(z)|^{-1} \leq 4^{D-1}C |w|^{-D+1} $. By similar arguments $|V^{\prime \prime}(t)| |V^\prime(t)|^{-2} \leq \widetilde{C} |w|^{-D} $ for some $\widetilde{C} > 0$.

Next, let us consider the case where $|\varphi(w)| > |\varphi(z)|$. Let us split the integral
\begin{align*}
K_V(z,w) = \int_z^{\varphi^{-1}(|\varphi(z)| \mathrm{e}^{\mathrm{i}\pi k /D})} \mathrm{e}^{V(t) - V(w)} \, \mathrm{d}t + \underbrace{\int_{\varphi^{-1}(|\varphi(z)| \mathrm{e}^{\mathrm{i}\pi k /D})}^w \mathrm{e}^{V(t) - V(w)} \, \mathrm{d}t}_{=: I_1} \, .
\end{align*}
For the first term we follow a similar procedure as before, and find that it may be bounded by something of the form $C | \widetilde{z} |^{1-D} |\mathrm{e}^{V(\widetilde{z}) - V(w)}|$ for some $C > 0$, where $\widetilde{z} := \varphi^{-1}(|\varphi(z)| \mathrm{e}^{\mathrm{i}\pi k /D})$. Note that by definition $|\varphi(z)| =|\varphi( \widetilde{z})|$. Then we may divide into two cases: (1) $|\varphi(z)| \geq \frac{1}{2}|\varphi(w)|$. In this case we may bound $|\mathrm{e}^{V(\widetilde{z}) - V(w)}| \leq 1$ and $|\widetilde{z}|^{1-D} \leq C |w|^{1-D}$ for some $C$. (2) $|\varphi(z)| \leq \frac{1}{2}|\varphi(w)|$. In this case, noting that $V(w)$ and $V(\widetilde{z})$ are both real and positive, we may bound $|\widetilde{z}|^{1-D}$ by a constant, and bound $\mathrm{e}^{V(\widetilde{z}) - V(w)} \leq \mathrm{e}^{-(1 - 2^{-D}) V(w)}$. Along the curve $\arg \varphi(w) = \pi k/D$ this goes to $0$ much faster than $|w|^{1-D}$ and so can certainly be bounded by this up to a constant.

For the second term, let us change variables $s = \varphi(t)$. Then we have
$$I_1 = \mathrm{e}^{-V(w)}\int_{|\varphi(z)|\mathrm{e}^{\mathrm{i}\pi k /D}}^{\varphi(w)} \mathrm{e}^{\gamma s^D} \, \frac{1}{\varphi^\prime(\varphi^{-1}(s))}\, \mathrm{d}s \, .$$
Note that in these new coordinates, the integration path is a straight line between the endpoints, and because these points have a common argument, a segment of the ray $\mathrm{e}^{\mathrm{i}\pi k /D}[0, +\infty)$. Again, on our neighbourhood of infinity, we have $| \varphi^\prime(\varphi^{-1}(s)) |^{-1} \leq C$ for some constant $C$. Hence
\begin{align*}
|I_1| \leq C \mathrm{e}^{-V(w)} \int_{|\varphi(z)|}^{|\varphi(w)|} \mathrm{e}^{\gamma s^D} \, \mathrm{d}s \leq  C 
\mathrm{e}^{-V(w)} \underbrace{\int_{\varrho}^{|\varphi(w)|} \mathrm{e}^{\gamma s^D} \, \mathrm{d}s}_{=: I_2} \, .
\end{align*}
Finally, we may perform an integration by parts on $I_2$ to get
\begin{align*}
I_2 &= \int_{\varrho}^{|\varphi(w)|} \frac{(D-1) s^{-D}}{\gamma D} \mathrm{e}^{\gamma s^D}\, \mathrm{d}s  + \frac{\mathrm{e}^{\gamma \varphi(w)^D}}{\gamma D |\varphi(w)|^{D-1}} - \frac{1}{\gamma D \varrho^{D-1}}\mathrm{e}^{\gamma \varrho^D} \\
&\leq \frac{(D-1) \varrho^{-D}}{\gamma D} I_2 + \frac{\mathrm{e}^{\gamma \varphi(w)^D}}{\gamma D |\varphi(w)|^{D-1}}  \, .
\end{align*}
We may choose $\varrho$ sufficiently large we can make $\frac{(D-1) \varrho^{-D}}{\gamma D} \leq \frac{1}{2}$, in which case we have
\begin{align*}
I_2 \leq  2\frac{\mathrm{e}^{\gamma \varphi(w)^D}}{\gamma D |\varphi(w)|^{D-1}}
\end{align*}
which proves the claim, upon observing that $|\varphi(w)| \geq \frac{1}{2}|w|$ and $\gamma \varphi(w)^D = V(w)$.
\end{proof}
Next, for $k \in [\![0, 2D-1 ]\!]$, consider the integral equation
\begin{align}\label{integralequation}
F_k(z) = 1+ \frac{1}{\lambda}\int_z^{\infty_k} K_{(-1)^k V}(z,w) F_k(w) \, \mathrm{d}w \, ,
\end{align}
for $F_k$ bounded on $\Omega_k$.
\begin{lemma} The integral equation \eqref{integralequation} has the unique solution for $z \in \Omega_k$
\begin{equation}\label{integralequationsolution}
\begin{split}
&F_k(z) = \\
&1 + \sum_{j=1}^{+\infty} \frac{1}{\lambda^j} \int^{\infty_k}_{z < w_j < w_{j-1} < \dots < w_1} K_{(-1)^k V}(z,w_j) K_{(-1)^k V}(w_j,w_{j-1}) \dots K_{(-1)^k V}(w_2, w_1) \, \mathrm{d}w_1 \dots  \mathrm{d}w_j
\end{split}
\end{equation}
where the contour of integration is $\Gamma_{k,z}$ and $a < b$ means that $a$ comes \enquote{before} $b$ according to the orientation of $\Gamma_{k,z}$. The series \eqref{integralequationsolution} converges absolutely and uniformly in $\Omega_k$, and in particular satisfies
\begin{align}\label{Fkbound}
F_k(z) = 1 + \mathcal{O}(|z|^{2-D}) &  & |z| \to +\infty \text{ for } z\in \Omega_k \, .
\end{align}
\end{lemma}
\begin{proof}
We first show that \eqref{integralequationsolution} converges absolutely and uniformly.
By \cref{KVestimate} we have
\begin{align*}
&1 + \sum_{j=1}^{+\infty} \frac{1}{\lambda^j} \int^{\infty_k}_{z < w_j < w_{j-1} < \dots < w_1} |K_{(-1)^k V}(z,w_j)| |K_{(-1)^k V}(w_j,w_{j-1})| \dots |K_{(-1)^k V}(w_2, w_1)| \, |\mathrm{d}w_1| \dots  |\mathrm{d}w_j| \\
&\leq  1 + \sum_{j=1}^{+\infty} \frac{1}{\lambda^j} C^j \int^{\infty_k}_{z < w_j < w_{j-1} < \dots < w_1} \prod_{p=1}^j \frac{1}{|w_p|^{D-1}} \, |\mathrm{d}w_1| \dots  |\mathrm{d}w_j| \\
&=  1 + \sum_{j=1}^{+\infty} \frac{1}{j!} C^j \Big( \frac{1}{\lambda^j}  \int_{\Gamma_{k,z}} \frac{1}{|w|^{D-1}} \, |\mathrm{d}w|  \Big)^j = \exp\Big( C \lambda^{-1}  \int_{\Gamma_{k,z}} \frac{1}{|w|^{D-1}} \, |\mathrm{d}w| \Big)  \, .
\end{align*}
Using $\varphi$ to change variables, and then rescaling the variables by $|\varphi(z)|$ we find an integral that is convergent and independent of $z$, hence
$$\int_{\Gamma_{k,z}} \frac{1}{|w|^{D-1}} \, |\mathrm{d}w| \leq \frac{\widetilde{C}}{|z|^{D-2}} \, . $$
This confirms that the series \eqref{integralequationsolution} is absolutely and uniformly convergent, and by the same method one easily sees that $F_k(z) = 1 +\mathcal{O}(|z|^{2-D})$. Finally, one need only substitute \eqref{integralequationsolution} into \eqref{integralequation} to see that it is indeed a solution.

Finally, we show uniqueness. Let $\widetilde{F}_k$ be another solution to \eqref{integralequation} and consider the difference $\delta_k = F_k - \widetilde{F}_k$. Then $\delta_k$ solves the integral equation
\begin{align*}
\delta_k(z) =  \frac{1}{\lambda}\int_z^{\infty_k} K_V(z,w) \delta_k(w) \, \mathrm{d}w \, .
\end{align*}
Let us take $z \in \Omega_k$. Hence, by \cref{KVestimate},
\begin{equation}\label{Deltainequality}
|\delta_{k}(z)| \leq \frac{C}{\lambda} \int_{z}^{\infty_k} |w|^{1-D} |\delta_k(w)| |\mathrm{d}w| \, .
\end{equation}
If we repeatedly substitute \eqref{Deltainequality} into itself we find 
\begin{align*}
|\delta_{k}(z)| &\leq ( C \lambda^{-1})^n \int_{z < w_1 < \dots < w_{n}}^{\infty_k} |w_1|^{1-D} | \dots |w_n|^{1-D} |\delta_k(w_n)| |\mathrm{d}w_1| \dots |\mathrm{d}w_n| \\
&= \frac{( C \lambda^{-1})^n}{(n-1)!} \int_z^{\infty_k} \Big( \int_{z}^w |\widetilde{w}|^{1-D} |\mathrm{d}\widetilde{w}| \Big)^{n-1} |w|^{1-D} |\delta_k(w)| \, |\mathrm{d}w| \\
&\leq \frac{( C \lambda^{-1})^n}{(n-1)!} \Big( \int_{z}^{\infty_k} |\widetilde{w}|^{1-D} \, |\mathrm{d}\widetilde{w} | \Big)^{n-1}  \int_z^{\infty_k} |w|^{1-D} |\delta_k(w)| \, |\mathrm{d}w| \, \overset{n \to +\infty}{\longrightarrow} 0 \, .
\end{align*}
This completes the proof of uniqueness.
\end{proof}
\begin{corollary}\label{entireness}
$F_k$ is analytic on the interior of $\Omega_k$ and extends to an entire function in $\mathbb{C}$, and solves \eqref{FkODE}.
\end{corollary}
\begin{proof}
Since the partial sums of \eqref{integralequationsolution} are clearly analytic, and the series converges uniformly on $\Omega_k$, the resulting function $F_k$ is analytic on the interior of $\Omega_k$. By differentiating \eqref{integralequation}, one sees that $F_k$ satisfies \eqref{FkODE}. This is a second order ODE in which every point in $\mathbb{C}$ is an ordinary point, hence by the existence and uniqueness theorem for linear ODEs in the complex domain (see Theorem 3.1. of \cite{olver1997asymptotics}, Theorem 6.1 and Proposition 6.1. of \cite{millerapplied}), the solution extends to an entire function on $\mathbb{C}$.
\end{proof}
\begin{remark} In fact, if we had worked a bit harder we could have defined a contour such that the series \eqref{integralequationsolution} converges for all $z \in \mathbb{C}$. However this will not be necessary for us.
\end{remark}
This concludes the proof of Proposition \ref{FkODEprop}.

\begin{proposition}[Sibuya solutions]\label{sibuya}
Define, for $k \in [\![ 0, 2D-1 ]\!]$
\begin{align}
Y_k(z) \overset{\mathrm{def}}{=} \mathrm{i}^k \mathrm{e}^{-\mathrm{i}\pi/4}   \begin{cases} \mathrm{e}^{-\frac{1}{2}V(z)}   F_k(z) & k \text{ even } \\
\lambda \,  \mathrm{e}^{\frac{1}{2}V(z)} F_k^\prime (z) & k \text{ odd } \, .
\end{cases}
\end{align}
Then $Y_k$ solves the second order ODE
\begin{align}\label{schroedinger}
Y_k^{\prime \prime}(z) = U(z) Y_k(z) \, \text{  where  } \,  U(z) = \frac{1}{4}V^\prime(z)^2 - \frac{1}{2}V^{\prime\prime}(z) + \lambda^{-1} \, .
\end{align}
\end{proposition}
\begin{proof}
Direct substitution.
\end{proof}
What we are calling \enquote{Sibuya} solutions are often also called \enquote{subdominant} or \enquote{recessive} solutions in the literature. It will turn out that these $Y_k$ are precisely the solutions obeying Liouville--Green asymptotics that we mentioned earlier.
\begin{remark}\leavevmode
\begin{enumerate}[label=(\arabic*)]
\item The prefactor $\mathrm{i}^k \mathrm{e}^{-\mathrm{i}\pi/4}$ is chosen so that later in our analysis the Wronskian of $Y_k$ and $Y_{k+1}$ will come out to be $1$.
\item We note that because $D$ is even, $2D$ is a multiple of $4$. Hence if we allow $k \in \mathbb{Z}$ and assume $F_k$ is periodic in $2D$ then $Y_k$ is also periodic in $2D$. However if $D$ had been odd then $Y_{k+2D} = - Y_k$. This will imply that
\begin{align}\label{cyclicity}
S_{2D} S_{2D-1} \dots S_{1} = (-1)^D \mathbb{I}
\end{align}
where $S_k$ is the $k$th Stokes matrix which will be defined in \eqref{stokesrelation}.
\item The formula for $Y_k$ for $k$ odd may appear strange and unmotivated to the reader. The motivation is the following. In order to have an integral equation that functions also as an asymptotic expansion, we needed to change the sign of $V$ for $k$ odd, which explains the $(-1)^k$ factor in \eqref{FkODE}. Next, if one has a solution $F$ to $$-V^\prime F^\prime - F^{\prime\prime} + \lambda^{-1}F = 0$$ then the equation has the natural integrating factor $\mathrm{e}^V$ so that it may be written $( \mathrm{e}^V F^\prime)^\prime = \lambda^{-1}\mathrm{e}^V F$. Thus if we introduce $G := \mathrm{e}^V F^\prime$ then we have the first order system for $(F,G)$
\begin{align*}
F^\prime = \mathrm{e}^{-V} G \, , & & G^\prime = \lambda^{-1} \mathrm{e}^V F \, .
\end{align*}
If we eliminate $G$ then we obtain our original ODE for $F$; however if we eliminate $F$ we obtain the ODE
$$V^\prime G^{\prime} - G^{\prime \prime} + \lambda^{-1}G = 0 \, .$$
This \enquote{duality} follows from the theory of quasi-derivatives in Sturm--Liouville theory (see Ch. 2 of \cite{zettl2005sturm}). Hence from a solution to \eqref{FkODE} we may construct, by duality, a solution to \eqref{Fequation}.
\end{enumerate}
\end{remark}
Thus $Y_k$ satisfies a time-independent Schrödinger equation with potential $\frac{1}{4}V^\prime(z)^2 - \frac{1}{2}V^{\prime\prime}(z)$ and spectral parameter $-\lambda^{-1}$. However we should emphasise that these functions $Y_k$ will, for generic $\lambda$, tend to diverge at $\infty_j$ for all $j \neq k$ and so in general cannot be regarded as eigenfunctions, though they could be regarded as \enquote{off shell} eigenfunctions.

Moving forward, we need the following useful bound.
\begin{lemma}\label{crudebound} Let $z \in \Omega_k$. Then $\int_{\Gamma_{k,z}} |\mathrm{e}^{(-1)^k \{ V(z) - V(w) \} }| | \mathrm{d}w| \leq C|z|$ for some constant $C$ and $|z|$ sufficiently large.
\end{lemma}
\begin{proof} For simplicity assume $k$ is even, since the proof is identical for $k$ odd except for reversing the sign of $V$.
\begin{align*}
\int_{\Gamma_{k,z}} |\mathrm{e}^{V(z) - V(w)}| | \mathrm{d}w| = \underbrace{\int_{z}^{\varphi^{-1}(|\varphi(z)| \mathrm{e}^{\mathrm{i}\pi k /D})} |\mathrm{e}^{V(z) - V(w)}| | \mathrm{d}w|}_{=: J_1} + \underbrace{\int_{\varphi^{-1}(|\varphi(z)| \mathrm{e}^{\mathrm{i}\pi k /D})}^{\infty_k} |\mathrm{e}^{V(z) - V(w)}| | \mathrm{d}w| }_{=: J_2} \, .
\end{align*}
Let us change variables to $\varphi(w)$, for which the Jacobian is bounded by a constant. Hence for $J_1$ we observe that the integrand has modulus at most $1$ and the path has arc length at most $2\pi |\varphi(z)|$. Hence $J_1 \leq C|z|$. For $J_2$ we perform an integration by parts
\begin{align*}
J_2 \leq C |\mathrm{e}^{V(z)}| \int_{|\varphi(z)|}^{+\infty} \mathrm{e}^{-\gamma s^D} \, \mathrm{d}s = - C |\mathrm{e}^{V(z)}| \int_{|\varphi(z)|}^{+\infty} \frac{1}{D\gamma s^{D-1}} \frac{\mathrm{d}}{\mathrm{d}s}( \mathrm{e}^{-\gamma s^{D}}) \, \mathrm{d}s \\
\leq \frac{C}{D\gamma |\varphi(z)|^{D-1}} + C |\mathrm{e}^{V(z)}| \int_{|\varphi(z)|}^{+\infty} \underbrace{\frac{\mathrm{d}}{\mathrm{d}s}\big( \frac{1}{D\gamma s^{D-1}} \big)}_{\leq 0} \mathrm{e}^{-\gamma s^{D}}\, \mathrm{d}s \, .
\end{align*}
Hence $J_2 \leq \widetilde{C}|z|^{1-D}$ for some $\widetilde{C} > 0$ which completes the proof.
\end{proof}
This is a very crude estimate but it suffices for our purposes.
\begin{lemma}\label{Fkprimebound} For $z \in \Omega_k$, we have $F_k^\prime(z) = (-1)^{k+1} \lambda^{-1} \frac{1}{V^\prime(z)}(1 + \mathcal{O}(|z|^{-D+2}))$ as $|z| \to +\infty$. 
\end{lemma}
\begin{proof}
Assume $k$ is even, since otherwise the analysis is the same except for reversing the sign of $V$. Our method is to begin with the bound on $F_k$ and progressively bootstrap to better estimates on $F_k^\prime$. Then differentiating \eqref{integralequation} we find
\begin{equation}\label{Fprimeformula}
F_k^\prime(z) = - \lambda^{-1} \int_z^{\infty_k} \mathrm{e}^{V(z) - V(w)} F_k(w) \, \mathrm{d}w \, .
\end{equation}
From \cref{crudebound} we then have $|F_k^\prime(z)| \leq C|z|$. Moving forward, let us then perform an integration by parts on \eqref{Fprimeformula} to find
\begin{align*}
F_k^\prime(z) &= -\lambda^{-1} \frac{1}{V^\prime(z)} F_k(z) - \lambda^{-1} \int_z^{\infty_k} \mathrm{e}^{V(z) - V(w)} F_k^\prime(w) \frac{1}{V^\prime(w)} \, \mathrm{d}w \\
&\quad - \lambda^{-1} \int_z^{\infty_k} \mathrm{e}^{V(z) - V(w)} F_k(w) \frac{\mathrm{d}}{\mathrm{d}w} \big( \frac{1}{V^\prime(w)} \big) \, \mathrm{d}w \, .
\end{align*}
Then applying our bound for $F^\prime_k$ to the right hand side, we see that $|F_k^\prime(z)| \leq C |z|^{-D+3} \leq C $ for some $C > 0$, since $D \geq 3$. Substituting our bound in again we find $|F_k^\prime(z)| \leq C |z|^{-D+2} \leq C|z|^{-1} $. Substituting in again, we arrive at the optimal scale which is $|F_k^\prime(z)| \leq C|z|^{-D+1} $.

Using this optimal bound, let us analyse each term. Clearly $$-\lambda^{-1} \frac{1}{V^\prime(z)} F_k(z) = -\lambda^{-1} \frac{1}{V^\prime(z)} \Big(1 + \mathcal{O}(|z|^{2-D})\Big)\, . $$ Similarly
$$- \lambda^{-1} \int_z^{\infty_k} \mathrm{e}^{V(z) - V(w)} F_k^\prime(w) \frac{1}{V^\prime(w)} \, \mathrm{d}w = \mathcal{O}(|z|^{3-2D}) \, . $$
This leaves
\begin{align*}
I &:=  \int_z^{\infty_k} \mathrm{e}^{V(z) - V(w)} F_k(w) \frac{\mathrm{d}}{\mathrm{d}w} \big( \frac{1}{V^\prime(w)} \big) \, \mathrm{d}w  \\
&= \underbrace{\int_z^{\infty_k} \mathrm{e}^{V(z) - V(w)}  \frac{\mathrm{d}}{\mathrm{d}w} \big( \frac{1}{V^\prime(w)} \big) \, \mathrm{d}w}_{=: I_1} + \underbrace{\int_z^{\infty_k} \mathrm{e}^{V(z) - V(w)} (F_k(w)-1) \frac{\mathrm{d}}{\mathrm{d}w} \big( \frac{1}{V^\prime(w)} \big) \, \mathrm{d}w}_{=: I_2} \, .
\end{align*}
For $I_1$ we perform yet another integration by parts.
\begin{align*}
I_1 &= - \int_z^{\infty_k} \frac{\mathrm{d}}{\mathrm{d}w} \big(\mathrm{e}^{V(z) - V(w)} \big) \frac{1}{V^\prime(w)}   \frac{\mathrm{d}}{\mathrm{d}w} \big( \frac{1}{V^\prime(w)} \big) \, \mathrm{d}w \\
&= \frac{1}{V^\prime(z)}   \frac{\mathrm{d}}{\mathrm{d}z} \big( \frac{1}{V^\prime(z)} \big) + \int_z^{\infty_k} \mathrm{e}^{V(z) - V(w)}  \frac{\mathrm{d}}{\mathrm{d}w} \big(  \frac{1}{V^\prime(w)}   \frac{\mathrm{d}}{\mathrm{d}w} \big( \frac{1}{V^\prime(w)} \big) \big) \, \mathrm{d}w \, .
\end{align*}
We do not need to calculate $\frac{1}{V^\prime(z)}   \frac{\mathrm{d}}{\mathrm{d}z} \big( \frac{1}{V^\prime(z)} \big)$ and $\frac{\mathrm{d}}{\mathrm{d}w} \big(  \frac{1}{V^\prime(w)}   \frac{\mathrm{d}}{\mathrm{d}w} \big( \frac{1}{V^\prime(w)} \big) \big) $ explicitly, except to observe that they are rational functions which scale like $\mathcal{O}(|z|^{1-2D})$ and $\mathcal{O}(|w|^{-2D})$ respectively. Hence $I_1 = \mathcal{O}(|z|^{1-2D})$. For $I_2$ we use that $F_k(w)  - 1 = \mathcal{O}(|w|^{2-D})$. This yields $I_2 = \mathcal{O}(|z|^{3-2D})$ which completes the proof.
\end{proof}
\begin{corollary}\label{liouvillegreen} For $z \in \Omega_k$ we have the asymptotics as $|z| \to +\infty$
\begin{align}
Y_k(z) = (1+ \mathcal{O}(|z|^{2-D})) \mathrm{i}^k \mathrm{e}^{-\mathrm{i}\pi/4}   \begin{cases} \mathrm{e}^{-\frac{1}{2}V(z)}   & k \text{ even } \\
 \mathrm{e}^{\frac{1}{2}V(z)} \frac{1}{V^\prime(z)}  & k \text{ odd } \, .
\end{cases}
\end{align}
\end{corollary}
Readers will notice that, up to a constant pre-factor, these asymptotics represent the Liouville--Green (WKB) asymptotics of $Y_k$ that we predicted earlier. We note that $\Omega_k \cap \Omega_{k+1} \neq \emptyset$ and $Y_k$ and $Y_{k+1}$ obey very different asymptotics on this overlap region, with one solution decaying and the other blowing up.
\begin{lemma} Define the Wronskian of two entire functions $f$ and $g$ as
\begin{align*}
W[f,g] \overset{\mathrm{def}}{=} \det \begin{pmatrix}
f & g \\
f^\prime & g^\prime
\end{pmatrix} \, .
\end{align*}
Then $W[Y_k, Y_j]$ is a constant for all $k,j \in [\![0, 2D-1]\!]$, and $W[Y_k, Y_{k+1}] \equiv 1$ (where the index $k$ is understood modulo $2D$).
\end{lemma}
\begin{proof}
$Y_k$ and $Y_{j}$ are both entire by \cref{entireness}, and hence so is $W[Y_k, Y_{j}]$. Since both $Y_k$ and $Y_{j}$ satisfy \eqref{schroedinger}, we have $\frac{\mathrm{d}}{\mathrm{d}z}W[Y_k, Y_{j}] = 0$. Hence $W[Y_k, Y_{j}]$ is a constant. Taking $j = k+1$, we may find this constant by computing the asymptotics of $W[Y_k, Y_{k+1}]$ in some direction in $\mathbb{C}$. Let us let $z \to \infty$ for $z \in \Omega_k \cap \Omega_{k+1}$. For simplicity, let us assume $k$ is even. Re-expressing $Y_k$, $Y_{k+1}$ in terms of $F_k $ and $F_{k+1}^\prime$, we find that
$$W[Y_k, Y_{k+1} ]  =  \lambda \det \begin{pmatrix}
F_k  & F_{k+1}^\prime \\ - \frac{1}{2}V^\prime F_k + F_k^\prime & \frac{1}{2}V^\prime F_{k+1}^\prime + F_{k+1}^{\prime \prime}
\end{pmatrix} \, . $$
Then using that $F_{k+1}$ satisfies \eqref{FkODE}, we find that
$$W[Y_k, Y_{k+1} ]  =  \lambda \det \begin{pmatrix}
F_k  & F_{k+1}^\prime \\ - \frac{1}{2}V^\prime F_k +  F_k^\prime & -\frac{1}{2}V^\prime F_{k+1}^\prime + \lambda^{-1} F_{k+1}
\end{pmatrix} = \lambda \det \begin{pmatrix}
F_k  & F_{k+1}^\prime \\  F_k^\prime &  \lambda^{-1} F_{k+1}
\end{pmatrix}  \, . $$
Then by \cref{FkODEprop} and \cref{Fkprimebound}, we see that $W[Y_k, Y_{k+1} ] = 1 + \mathcal{O}(|z|^{2 - 2D})$ for $z \in \Omega_k \cap \Omega_{k+1}$. $\Omega_k \cap \Omega_{k+1}$ is nonempty and unbounded, hence we may send $|z| \to +\infty$ and conclude that $W[Y_k, Y_{k+1} ] \equiv 1$. A similar argument holds for $k$ odd.
\end{proof}
\begin{corollary}\label{stokesrelation} There is a constant $\sigma_k \in \mathbb{C}$ such that $Y_{k+1} = \sigma_k Y_k - Y_{k-1}$.
\end{corollary}
\begin{proof}
From $W[Y_{k-1},Y_k] = 1$ it follows that $Y_{k-1}$ and $Y_k$ are linearly independent. Hence they span the space of solutions to \eqref{schroedinger}. Hence $Y_{k+1} = \sigma_k Y_k + \widetilde{\sigma}_k Y_{k-1}$ for some constants $\sigma_k, \widetilde{\sigma}_k \in \mathbb{C}$. Hence
$$-1 = W[Y_{k+1}, Y_k ] = W[\sigma_k Y_k + \widetilde{\sigma}_k Y_{k-1}, Y_k ] = \widetilde{\sigma}_k \, .$$
\end{proof}
\begin{remark} If $\sigma_k = 0$, then $Y_{k+1} = - Y_{k-1}$ rapidly decays both as $z \to \infty_{k+1}$ and $z \to \infty_{k-1}$. Hence $Y_{k+1}$ is an \enquote{eigenfunction} for $ - \frac{\mathrm{d}^2}{\mathrm{d}z^2} + U_0(z)$ on the Hilbert space associated to the directions $k-1$ and $k+1$. This indicates that $\sigma_k$ should be regarded as a spectral quantity.
\end{remark}
The relation in \cref{stokesrelation} may be represented as
\begin{align}\label{stokesrelation2}
\begin{pmatrix}
Y_{k+1} \\
Y_k
\end{pmatrix} = \underbrace{\begin{pmatrix}
\sigma_k & -1 \\ 1 & 0
\end{pmatrix}}_{S_k} \begin{pmatrix}
Y_k \\ Y_{k-1}
\end{pmatrix} \, .
\end{align}
\begin{lemma}\label{Mlemma} Define $M := S_D S_{D-1} \dots S_1$. Then $\det M = 1$ and $W[Y_0,Y_D] = M_{21} \neq 0$ for $\lambda > 0$.
\end{lemma}
\begin{proof}
$\det M = 1$ is trivial, since $\det S_k = 1$.
Next, by repeatedly applying the relation \eqref{stokesrelation2} we have
\begin{align}\label{Mrelation}
\begin{pmatrix}
Y_{D+1} \\
Y_D
\end{pmatrix} = \begin{pmatrix}
M_{11} & M_{12} \\ M_{21} & M_{22}
\end{pmatrix} \begin{pmatrix}
Y_1 \\ Y_{0}
\end{pmatrix} \, .
\end{align}
Suppose by way of contradiction that $M_{21} = 0$. Then we would have $Y_{D}(z) = M_{22} Y_0(z)$. $1 = \det M = M_{11} M_{22}$, and so $M_{22} \neq 0$. Hence $Y_D$ decays rapidly to zero as $z \to \infty_0 = +\infty$ and as $z \to \infty_D = -\infty$. Let $U_0(z) = \frac{1}{4}V^\prime(z)^2 - \frac{1}{2}V^{\prime\prime}(z)$. Then by \eqref{schroedinger} we have
$$\int_{\mathbb{R}} \overline{Y}_D(z) (Y_D^{\prime \prime}(z) - U_0(z)Y_D(z)) \, \mathrm{d}z = \lambda^{-1}\int_{\mathbb{R}} |Y_D(z)|^2 \, \mathrm{d}z > 0 \, .$$
Note that this integral is finite because of the decay properties of $Y_D$ (\cref{liouvillegreen}). Next, we observe by the factorisation \eqref{factorisation} we may integrate by parts, hence
$$\int_{\mathbb{R}} \overline{Y}_D(z) (Y_D^{\prime \prime}(z) - U_0(z)Y_D(z)) \, \mathrm{d}z = -\int_{\mathbb{R}} \Big| Y_D^{ \prime}(z) + \frac{1}{2}V^\prime(z) Y_D(z)  \Big|^2 \, \mathrm{d}z \leq 0 $$
which is a contradiction. Finally, by \eqref{Mrelation}, we have $Y_D = M_{21} Y_1 + M_{22}Y_0$. Hence
$$W[Y_0, Y_D] = M_{21} \underbrace{W[Y_0, Y_1]}_{= 1} + M_{22}\underbrace{W[Y_0, Y_0]}_{ = 0 } = M_{21}\, .$$
\end{proof}
\begin{lemma}\label{conjugationsymmetry} Let $\overline{Y_k}(z) := \overline{Y_k(\overline{z})}$. We have the relations $\overline{\sigma_k} = \sigma_{-k}$ and $\overline{Y_k} = \mathrm{i} Y_{-k}$, where the index $k$ is understood modulo $2D$. In particular, $Y_0$ and $Y_D$ have argument $- \pi/4 + \pi \mathbb{Z}$, while $\sigma_0$ and $\sigma_D$ are purely real.
\end{lemma}
\begin{proof}
$F_k$ is uniquely defined by the fact that it solves the ODE \eqref{FkODE} and has the asymptotics $F_k(z) = 1 + \mathcal{O}(|z|^{2-D})$. Since the polynomial $V$ has only real coefficients, and $(-1)^k = (-1)^{-k}$, we have that $\overline{F_k}(z) = F_{-k}(z)$. Inserting this into \cref{sibuya} we see that $\overline{Y_k}(z) = \mathrm{i} Y_{-k}(z)$. Finally, we observe that $\sigma_k = W[Y_{k-1},Y_{k+1}]$. Hence $\overline{\sigma_k} = -W[Y_{-k+1},Y_{-k-1}] = W[Y_{-k-1},Y_{-k+1}] = \sigma_{-k}$.
\end{proof}

\subsection{The inhomogeneous problem}

This concludes the discussion of the homogeneous problem. Let us now return to the problem with the source term, which is what interested us to begin with. In particular, let us consider the second order ODE equation for $\Phi$
\begin{align}\label{schroedingersource}
&\Phi^{\prime\prime}(z) = U(z) \Phi(z) + \widetilde{\Psi}(z)  \; \text{where } \;
U(z) = \frac{1}{4}V^\prime(z)^2 - \frac{1}{2}V^{\prime\prime}(z) + \lambda^{-1} \, .
\end{align}
The question which interests us is this: given $\widetilde{\Psi}(z) = Q(z) \mathrm{e}^{-\frac{1}{2}V(z)}$ for a polynomial $Q$, when does there exist a solution to \eqref{schroedingersource} of the form $\Phi(z) = P(z) \mathrm{e}^{-\frac{1}{2}V(z)}$ for some polynomial $P$? We now proceed to answer this question.

To begin, we observe that we can explicitly write the general solution.
\begin{lemma}
Let $k \in [\![0, 2D-1 ]\!]$. Any solution of \eqref{schroedingersource} admits the representation
\begin{align}\label{generalsolution}
\Phi(z) = \int_0^z \det \begin{pmatrix}
Y_{k+1}(z) & Y_{k+1}(w) \\ Y_{k}(z) & Y_{k}(w)
\end{pmatrix} \, \widetilde{\Psi}(w) \, \mathrm{d}w + A_k Y_{k+1}(z) + B_k Y_k(z)
\end{align}
for some $A_k, B_k \in \mathbb{C}$.
\end{lemma}
\begin{proof}
Direct calculation shows that
$$z \mapsto \int_0^z \det \begin{pmatrix}
Y_{k+1}(z) & Y_{k+1}(w) \\ Y_{k}(z) & Y_{k}(w)
\end{pmatrix} \, \widetilde{\Psi}(w) \, \mathrm{d}w$$
is a particular solution of \eqref{schroedingersource}. Since $Y_k$ and $Y_{k+1}$ span the basis of solutions of \eqref{schroedinger}, the general solution is of the form \eqref{generalsolution}.
\end{proof}
\begin{lemma}\label{kernelindependence} The kernel
$$(z,w) \mapsto \det \begin{pmatrix}
Y_{k+1}(z) & Y_{k+1}(w) \\ Y_{k}(z) & Y_{k}(w)
\end{pmatrix}$$
is independent of $k$. Furthermore, for all $k \in [\![0, 2D-1]\!]$,
\begin{align}\label{kernelrelation}
\det \begin{pmatrix}
Y_{k+1}(z) & Y_{k+1}(w) \\ Y_{k}(z) & Y_{k}(w)
\end{pmatrix} = \frac{1}{W[Y_0,Y_D]} \det \begin{pmatrix}
Y_{D}(z) & Y_{D}(w) \\ Y_{0}(z) & Y_{0}(w)
\end{pmatrix} \, .
\end{align}
\end{lemma}
\begin{proof}
The first claim follows from the fact that $\det S_k = 1$. Hence, by taking $k = 0$, we may write
$$\det \begin{pmatrix}
Y_{k+1}(z) & Y_{k+1}(w) \\ Y_{k}(z) & Y_{k}(w)
\end{pmatrix} =  \det \begin{pmatrix}
Y_{1}(z) & Y_{1}(w) \\ Y_{0}(z) & Y_{0}(w)
\end{pmatrix} \, .$$
Next, since $Y_0$ and $Y_1$ form a basis of solutions of \eqref{schroedinger}, we must have $Y_D = a Y_1 + b Y_0$ for some constants $a$ and $b$. Taking Wronskians of both sides we see that $a = W[Y_0, Y_D]$, $b = -W[Y_1, Y_D]$. Hence
$$\begin{pmatrix}
Y_D \\ Y_0
\end{pmatrix} = \begin{pmatrix}
a & b \\ 0 & 1
\end{pmatrix} \begin{pmatrix}
Y_1 \\ Y_0
\end{pmatrix} \, .$$
This transformation matrix has determinant $a = W[Y_0,Y_D]$, which proves the second claim.
\end{proof}

Now, of course, if one fixes a solution $\Phi$ of \eqref{schroedingersource}, the choice of $k$ in the representation \eqref{generalsolution} is arbitrary. Hence the right hand side of \eqref{generalsolution} is in fact independent of $k$. Since by \cref{kernelindependence}, the kernel is independent of $k$, $A_k Y_{k+1}(z) + B_k Y_k(z)$ must also be independent of $k$. Let us write $\mathsf{u}_k(z) = (Y_{k+1}(z) , Y_k(z) )^\mathsf{T}$ and $\mathsf{v}_k = (A_k, B_k)^\mathsf{T}$. Hence for all $k$ we have
\begin{align*}
\mathsf{u}_k(z)^\mathsf{T} \mathsf{v}_k = \mathsf{u}_{k+1}(z)^\mathsf{T} \mathsf{v}_{k+1} & & \forall k \in [\![0,2D-1]\!] \, . 
\end{align*}
Then using \eqref{stokesrelation2}, we have $\mathsf{u}_k(z) = S_k \mathsf{u}_{k-1}(z)$. Hence $\mathsf{u}_k(z)^\mathsf{T} \mathsf{v}_k = \mathsf{u}_{k}(z)^\mathsf{T} S_{k+1}^\mathsf{T} \mathsf{v}_{k+1}$. Then since $Y_k$ and $Y_{k+1}$ are linearly independent, we have
\begin{align}\label{vkrelation}
\mathsf{v}_k = S_{k+1}^\mathsf{T} \mathsf{v}_{k+1} \, .
\end{align}
Thus if we can determine $\mathsf{v}_k$ for \textit{one} value of $k$, it is then determined for \textit{all} values of $k$. This is natural because \eqref{schroedingersource} is a second order ODE so we expect two constants of integration.

We now arrive at the central proposition of this article.
\begin{proposition}\label{polynomialcondition} Let $\widetilde{\Psi}(z) = Q(z) \mathrm{e}^{-\frac{1}{2}V(z)}$ where $Q$ is a polynomial. Then the following are equivalent.
\begin{enumerate}[label=(\arabic*)]
\item \eqref{schroedingersource} admits a solution of the form $\Phi(z) = P(z) \mathrm{e}^{-\frac{1}{2}V(z)}$ for some polynomial $P$.
\item For all $j \in [\![1, D-1 ]\!] \setminus \{D/2 \} $
\begin{align}\label{linearcondition}
\int_0^{\infty_{2j}} Y_{2j}(z) \widetilde{\Psi}(z) \, \mathrm{d}z - \beta_j \int_0^{\infty_0} Y_{0}(z) \widetilde{\Psi}(z) \, \mathrm{d}z - \kappa_j \int_0^{\infty_D} Y_{D}(z) \widetilde{\Psi}(z) \, \mathrm{d}z = 0 
\end{align}
where $\beta_j = -\frac{W[Y_D, Y_{2j}]}{W[Y_0,Y_D]}$ and $\kappa_j = \frac{W[Y_0, Y_{2j}]}{W[Y_0,Y_D]}$ are constants.
\end{enumerate}
\end{proposition}
\begin{proof}
$\underline{(1) \implies (2)}$ We suppose that $\widetilde{\Psi}(z) = \Phi^{\prime\prime}(z) - U(z) \Phi(z)$ where $\Phi(z)  = P(z) \mathrm{e}^{-\frac{1}{2}V(z)}$ for $P$ a polynomial, and let us substitute this into the left hand side of \eqref{linearcondition}. To this end, let us note the following important identity. Let $A$ and $B$ be two analytic functions decaying sufficiently fast at infinity so that the relevant integrals converge. Then if we integrate by parts twice
$$\int_0^{\infty_k} A(z) B^{\prime \prime}(z) \, \mathrm{d}z = \int_0^{\infty_k} A^{\prime \prime}(z) B(z) \, \mathrm{d}z + W[B,A](0) - \lim_{z \to \infty_k} W[B,A](z)  $$
assuming $\lim_{z \to \infty_k} W[B,A](z)$ exists. Let us observe that
\begin{align*}
W[\Phi, Y_{2j}](z) = (-1)^j \mathrm{e}^{-\mathrm{i}\pi/4}\mathrm{e}^{-V(z)} \det \begin{pmatrix}
 P(z) & F_{2j}(z) \\ P^\prime(z) & F_{2j}^\prime(z)
\end{pmatrix} \longrightarrow 0 & &\text{ as } z \to \infty_{2j}
\end{align*}
since the determinant grows at most polynomially. Applying this to our case, we find that the left hand side of \eqref{linearcondition} gives the Wronskian
\begin{align*}
W\Big[ \Phi  ,  Y_{2j} - \beta_j Y_0 - \kappa_j Y_D   \Big] (0) \, .
\end{align*}
We claim that
\begin{equation}\label{identity}
Y_{2j} - \beta_j Y_0 - \kappa_j Y_D \equiv 0
\end{equation}
which would prove our claim. To see this, observe that since $W[Y_0, Y_D] \neq 0$, $Y_0$ and $Y_D$ form a basis of solutions of \eqref{schroedinger}. Hence $Y_{2j} = \widetilde{\beta}_j Y_0 + \widetilde{\kappa}_j Y_D$ for some $\widetilde{\beta}_j, \widetilde{\kappa}_j \in \mathbb{C}$. Taking Wronskians of both sides we see that
\begin{align*}
W[Y_D, Y_{2j}] &= \widetilde{\beta}_j W[Y_D, Y_0]   \\
W[Y_0, Y_{2j}] &=\widetilde{\kappa}_j W[Y_0, Y_D]  \, .
\end{align*}
Hence $\widetilde{\beta}_j = \beta_j$ and $\widetilde{\kappa}_j = \kappa_j$, and so we deduce \eqref{identity}.

$\underline{(2) \implies (1)}$ To prove the reverse claim, we must find a solution of polynomial type. We introduce the following convenient notation
\begin{align*}
\mathcal{I}_j[\widetilde{\Psi}] := \int_0^{\infty_{2j}}Y_{2j}(w)  \widetilde{\Psi}(w)  \, \mathrm{d}w \, , & & j \in [\![0, D-1 ]\!] \, .
\end{align*}
We propose the following
$$\Phi(z) := \frac{1}{W[Y_0, Y_D]} \int_0^z \det \begin{pmatrix}
Y_{D}(z) & Y_{D}(w) \\ Y_{0}(z) & Y_{0}(w)
\end{pmatrix} \widetilde{\Psi}(w) \, \mathrm{d}w - \frac{\mathcal{I}_0[\widetilde{\Psi}]}{W[Y_0, Y_D]} Y_D(z) + \frac{\mathcal{I}_{D/2}[\widetilde{\Psi}]}{W[Y_0, Y_D]} Y_0(z) \, . $$
We claim that $\Phi(z) = P(z) \mathrm{e}^{-\frac{1}{2}V(z)}$ for some polynomial $P$.  Let $j \in [\![0, D-1]\!]$. From \eqref{kernelrelation}, and the fact that $Y_{2j+1}$ and $Y_{2j}$ form a basis of solutions to \eqref{schroedinger}, we have
\begin{equation}\label{sectorformula}
\begin{split}
&\Phi(z) =  \int_0^z \det \begin{pmatrix}
Y_{2j+1}(z) & Y_{2j+1}(w) \\ Y_{2j}(z) & Y_{2j}(w)
\end{pmatrix} \widetilde{\Psi}(w) \, \mathrm{d}w + A_{2j} Y_{2j+1}(z) + B_{2j} Y_{2j}(z) \\
&\text{i.e.} \; - \frac{\mathcal{I}_0[\widetilde{\Psi}]}{W[Y_0, Y_D]} Y_D(z) + \frac{\mathcal{I}_{D/2}[\widetilde{\Psi}]}{W[Y_0, Y_D]} Y_0(z) = A_{2j} Y_{2j+1}(z) + B_{2j} Y_{2j}(z)
\end{split}
\end{equation}
for some unique constants $A_{2j}, B_{2j} \in \mathbb{C}$. We claim that, given the relation \eqref{linearcondition}, we have $A_{2j} = - B_{2j-1} = - \mathcal{I}_{j}[\widetilde{\Psi}]$ for all $j \in [\![ 0, D-1 ]\!]$. To see this, first observe that
\begin{align*}
Y_D &= W[Y_{2j}, Y_D] \,  Y_{2j+1} - W[Y_{2j+1}, Y_D] \,  Y_{2j} \\
Y_0 &= W[Y_{2j}, Y_0] \,  Y_{2j+1} - W[Y_{2j+1}, Y_0] \,  Y_{2j} \, .
\end{align*}
Collecting terms, we see that
\begin{align}\label{A2jformula}
A_{2j} = - W[Y_{2j}, Y_D]\frac{\mathcal{I}_0[\widetilde{\Psi}]}{W[Y_0, Y_D]} + W[Y_{2j}, Y_0] \frac{\mathcal{I}_{D/2}[\widetilde{\Psi}]}{W[Y_0, Y_D]} = -\beta_{j} \mathcal{I}_0[\widetilde{\Psi}] - \kappa_{j}  \mathcal{I}_{D/2}[\widetilde{\Psi}] \, .
\end{align}
We observe that \eqref{linearcondition} is trivially true for $j = 0$ and $j = D/2$, hence we may add these cases \enquote{for free.} Applying \eqref{linearcondition} to the above equation, we see that $A_{2j} = - \mathcal{I}_{j}[\widetilde{\Psi}]$ for $j \in [\![ 0, D-1]\!]$. Finally, by \eqref{vkrelation}, we observe that $B_{2j-1} = - A_{2j}$.

We now claim this implies that $\Phi(z) = P(z) \mathrm{e}^{-\frac{1}{2}V(z)}$ for some polynomial. We argue this by Liouville's theorem. Namely, $\mathrm{e}^{\frac{1}{2}V(z)}\Phi(z) $ is clearly an entire function, since the Sibuya solutions $Y_j$ are entire, and we claim that
\begin{align}\label{polynomialbound}
|\mathrm{e}^{\frac{1}{2}V(z)}\Phi(z)| \leq C |z|^{\theta} & & \text{ for some constants } \, \theta , C> 0 \text{ and for } |z| \text{ sufficiently large.}
\end{align}
By Liouville's theorem, this would imply that $\mathrm{e}^{\frac{1}{2}V(z)}\Phi(z)$ is a polynomial. We argue this sector by sector. Clearly, there is a $\widetilde{\varrho}$ sufficiently large so that 
$$\{ z \in \mathbb{C} \, : \, |z| > \widetilde{\varrho} \} \subset \bigcup_{k=0}^{2D-1} (\Omega_{k} \cap \Omega_{k+1}) \, .$$
Hence we claim that for each $k \in [\![0, 2D-1]\!]$, the bound \eqref{polynomialbound} holds in sector $\Omega_{k} \cap \Omega_{k+1}$, and since there are only finitely many such sectors, we obtain \eqref{polynomialbound} in the whole plane, outside of a compact set. 

Let us suppose to begin with that $k = 2j$ is even. Let us re-write \eqref{sectorformula} as 
\begin{align*}
\mathrm{e}^{\frac{1}{2}V(z)}\Phi(z) &= \mathrm{e}^{\frac{1}{2}V(z)} Y_{2j+1}(z) \underbrace{\Big[ \mathcal{I}_j[\widetilde{\Psi}] + A_{2j} \Big]}_{=0} - \mathrm{e}^{\frac{1}{2}V(z)} Y_{2j}(z) \int_0^z Y_{2j+1}(w) \widetilde{\Psi}(w) \, \mathrm{d}w  + B_{2j} \,  \mathrm{e}^{\frac{1}{2}V(z)} Y_{2j}(z) \\
&\quad - \mathrm{e}^{\frac{1}{2}V(z)} Y_{2j+1}(z) \int_{z}^{\infty_{2j}} Y_{2j}(w) \widetilde{\Psi}(w) \, \mathrm{d}w \, .
\end{align*}
We show that each term on the right hand side has at most polynomial growth. Recall that $\widetilde{\Psi}(z) = Q(z) \mathrm{e}^{-\frac{1}{2}V(z)}$ for $Q$ a polynomial. By \cref{liouvillegreen},  $|\mathrm{e}^{\frac{1}{2}V(z)} Y_{2j}(z)| \leq C$. Likewise by \cref{liouvillegreen},
\begin{align*}
|Y_{2j+1}(w) \widetilde{\Psi}(w) |\leq C \frac{|Q(w)|}{(1+|w|)^{D-1}} \leq \widetilde{C} (1+|w|)^{\deg Q - D+1} \text{ for some constants } C, \widetilde{C} > 0 \, .
\end{align*}
Taking the straight line path we then have
$$\Big|\int_0^z Y_{2j+1}(w) \widetilde{\Psi}(w) \, \mathrm{d}w \Big| \leq \widetilde{C} |z| (1+|z|)^{\max\{0, \deg Q - D+1\} } \, .$$
This leaves only the final term. To this end, let us use the decomposition from \cref{polynomialdecomposition}.
\begin{align*}
\int_{z}^{\infty_{2j}} Y_{2j}(w) \widetilde{\Psi}(w) \, \mathrm{d}w = \underbrace{\int_{z}^{\infty_{2j}} Y_{2j}(w) \widetilde{\Psi}_0(w) \, \mathrm{d}w}_{\Circled{A}} + \underbrace{\int_{z}^{\infty_{2j}} Y_{2j}(w) \Big\{  \widetilde{\Psi}_1^{\prime\prime}(w) - U(w) \widetilde{\Psi}_1(w) \Big\}\, \mathrm{d}w }_{ \Circled{B}} \, ,
\end{align*}
where $\widetilde{\Psi}_i(w) = Q_i(w) \mathrm{e}^{-\frac{1}{2}V(w)}$ for $i \in \{ 0,1 \}$, for some polynomials $Q_0, Q_1$, and where $\deg Q_0 \leq D-1$.

With regard to $\Circled{A}$, let us integrate by parts 
\begin{align*}
\Circled{A} &= -\int_{z}^{\infty_{2j}} V^\prime(w)^{-1} F_{2j}(w) Q_0(w)  \frac{\mathrm{d}}{\mathrm{d}w} ( \mathrm{e}^{-V(w)} ) \, \mathrm{d}w \\
&=  V^\prime(z)^{-1} F_{2j}(z) Q_0(z)\mathrm{e}^{-V(z)} - \underbrace{\int_z^{\infty_{2j}} \frac{V^{\prime\prime}(w)}{V^\prime(w)^2} F_{2j}(w) Q_0(w) \mathrm{e}^{-V(w)} \, \mathrm{d}w}_{\Circled{1}} \\
&\quad + \underbrace{\int_z^{\infty_{2j}} \frac{1}{V^\prime(w)} F_{2j}^\prime(w) Q_0(w) \mathrm{e}^{-V(w)} \, \mathrm{d}w}_{\Circled{2}} + \underbrace{\int_z^{\infty_{2j}} \frac{1}{V^\prime(w)} F_{2j}(w) Q_0^\prime(w) \mathrm{e}^{-V(w)} \, \mathrm{d}w}_{\Circled{3}} \, .
\end{align*}
By \eqref{Fkbound}, $|V^\prime(z)^{-1} F_{2j}(z) Q_0(z)| \leq C$ for some constant $C$ for $z \in \Omega_{2j}$ and $|z|$ sufficiently large. Likewise, in integrals $\Circled{1}$ and $\Circled{3}$, $|\frac{V^{\prime\prime}(w)}{V^\prime(w)^2} Q_0(w)| \leq C|w|^{-1}$ and $|V^\prime(w)^{-1} Q_0^\prime(w) | \leq C|w|^{-1}$. Combining \eqref{Fkbound} and by \cref{crudebound},
$$| \, \Circled{1} \, | \, , \, | \, \Circled{3} \, | \leq C |\mathrm{e}^{-V(z)}| \, .$$
for some $C > 0$. Likewise by \cref{Fkprimebound} and \cref{crudebound}, $| \, \Circled{2} \, | \leq C |z|^{-D+2} |\mathrm{e}^{-V(z)}|$. Recall that we must multiply $\Circled{A}$ by $\mathrm{e}^{\frac{1}{2}V(z)} Y_{2j+1}(z)$ which by \cref{liouvillegreen} grows like $V^\prime(z)^{-1}\mathrm{e}^{V(z)}$. This establishes polynomial growth for these terms.

Finally, we consider $\Circled{B}$. By integrating by parts and then performing row operations on the Wronskian we may see
\begin{align*}
\Circled{B} = W[\widetilde{\Psi}_1, Y_{2j}](z) = (-1)^j \mathrm{e}^{-\mathrm{i}\pi/4} \mathrm{e}^{-V(z)} W[Q_1, F_{2j}](z) \, .
\end{align*}
By \eqref{Fkbound} and \cref{Fkprimebound}, we see that $W[Q_1, F_{2j}](z)$ has at most polynomial growth as $|z| \to +\infty$ for $z \in \Omega_{2j}$. Once again, $\Circled{B}$ must be multiplied by $\mathrm{e}^{\frac{1}{2}V(z)} Y_{2j+1}(z)$ which grows like $V^\prime(z)^{-1}\mathrm{e}^{V(z)}$. This completes the proof of polynomial growth for $k = 2j$ even.

For $k = 2j-1$ we may use the representation
\begin{align*}
\mathrm{e}^{\frac{1}{2}V(z)} \Phi(z) &= \mathrm{e}^{\frac{1}{2}V(z)} Y_{2j-1}(z) \underbrace{\Big[ B_{2j-1} - \mathcal{I}_{j}[\widetilde{\Psi}] \Big]}_{=0} + A_{2j-1} \mathrm{e}^{\frac{1}{2}V(z)} Y_{2j}(z)   \\
&\quad + \mathrm{e}^{\frac{1}{2}V(z)} Y_{2j-1}(z) \int_z^{\infty_{2j}}Y_{2j}(w) \widetilde{\Psi}(w) \, \mathrm{d}w  +  \mathrm{e}^{\frac{1}{2}V(z)} Y_{2j}(z) \int_0^z Y_{2j-1}(w) \widetilde{\Psi}(w) \, \mathrm{d}w \, .
\end{align*}
The proof then proceeds in the same way as in the even case.
\end{proof}
\begin{lemma}\label{polynomialdecomposition} Let $Q$ be a polynomial. Then there exist polynomials $Q_0, Q_1$, where $\deg Q_0 \leq D-1$, such that
\begin{align*}
Q = Q_0 + Q_1^{\prime\prime} - V^\prime Q_1^\prime - \lambda^{-1}Q_1 \, .
\end{align*}
Note that this implies that
\begin{align*}
Q(x) \mathrm{e}^{-\frac{1}{2}V(x)} = Q_0(x) \mathrm{e}^{-\frac{1}{2}V(x)} + \frac{\mathrm{d}^2}{\mathrm{d}x^2} \Big(Q_1(x) \mathrm{e}^{-\frac{1}{2}V(x)} \Big) - U(x) Q_1(x) \mathrm{e}^{-\frac{1}{2}V(x)}  \, .
\end{align*}
\end{lemma}
\begin{proof}
We prove this by induction on $n = \deg Q$. If $n \leq D-1$, then the claim is trivial since we may take $Q_0 := Q$ and $Q_1 := 0$. This establishes the base case.

Next, by the inductive hypothesis, suppose that the decomposition holds for polynomials of degree up to $n-1$ and let $Q$ be a polynomial of degree exactly $n$, where without loss of generality $n > D-1$. Let $Q$ have leading coefficient $\rho$. If we let $$S(x) := - \frac{\rho}{\gamma D(n-D+2)} x^{n-D+2}$$ then $S^{\prime\prime} - V^\prime S^\prime - \lambda^{-1} S$ is a polynomial of degree $n$ with leading coefficient $\rho$. Hence $$Q - \Big(S^{\prime\prime} - V^\prime S^\prime - \lambda^{-1} S\Big)$$ is a polynomial of degree at most $n-1$. By the inductive hypothesis, there exist polynomials $\widetilde{Q}_0, \widetilde{Q}_1$, where $\deg \widetilde{Q}_0 \leq D-1$, such that
$$Q - \Big(S^{\prime\prime} - V^\prime S^\prime - \lambda^{-1} S\Big) = \widetilde{Q}_0 + \widetilde{Q}_1^{\prime\prime} - V^\prime \widetilde{Q}_1^\prime - \lambda^{-1}\widetilde{Q}_1  \, .$$
The proof is then complete upon taking $Q_0 := \widetilde{Q}_0$ and $Q_1 := \widetilde{Q}_1 + S$.
\end{proof}
\begin{remark}\leavevmode
\begin{enumerate}[label=(\arabic*)]
\item In \eqref{linearcondition} we have fixed the lower endpoint to $0$, but we could have replaced it with any other point $z_0 \in \mathbb{C}$ (and indeed, this lower endpoint could depend on $j$). More precisely, 
$$z_0 \mapsto \int_{z_0}^{\infty_{2j}} Y_{2j}(z) \widetilde{\Psi}(z) \, \mathrm{d}z - \beta_j \int_{z_0}^{\infty_0} Y_{0}(z) \widetilde{\Psi}(z) \, \mathrm{d}z - \kappa_j \int_{z_0}^{\infty_D} Y_{D}(z) \widetilde{\Psi}(z) \, \mathrm{d}z$$
is independent of $z_0$ by virtue of \eqref{identity}.
\item Let us also remark that it is impossible for $\beta_j$ and $\kappa_j$ to vanish simultaneously, since if $\beta_j = \kappa_j =0$, then $W[Y_0, Y_{2j}] = W[Y_D, Y_{2j} ] =0$. However this would imply that $Y_{2j}$ is proportional to both $Y_0$ and $Y_D$. This would imply that $Y_0$ and $Y_D$ are proportional to each other, which we have already established cannot be the case for $\lambda > 0$ (see \cref{Mlemma}).
\item Using similar kind of reasoning as in \cref{Mlemma}, one can write $\beta_j$ and $\kappa_j$ in terms of the Stokes multipliers $\{ \sigma_k \}_{k \in [\![ 0, 2D-1]\!]}$. More precisely, $W[Y_0, Y_{2j}] = (S_{2j} S_{2j-1} \dots S_1)_{21}$ and $W[Y_D, Y_{2j}] = (S_{2j} S_{2j-1} \dots S_{D+1})_{21}$ (where by \eqref{cyclicity} we understand all indices modulo $2D$).
\end{enumerate}
\end{remark}
\begin{lemma}[Zeros of the dual polynomials] Assume $n \geq 1$. Then $\Psi_n$ has at least $n$ real zeroes (not counting multiplicity), of which at least $n - D/2 + 1$ must be simple.
\end{lemma}
\begin{proof}
Let $x_1, \dots, x_p \in \mathbb{R}$ be the points on the real line where $\Psi_n$ changes sign, i.e. the real zeroes of odd multiplicity. Hence $(x-x_1) \dots (x-x_p)\Psi_n(x)$ has the same sign throughout $\mathbb{R}$. Hence
$$\int_{\mathbb{R}}(x-x_1) \dots (x-x_p)\Psi_n(x) \mathrm{e}^{-V(x)}\, \mathrm{d}x \neq 0 \, .$$
However if $p \leq n-1$ then by \eqref{Psiorthog} the above integral is $0$. Hence we must have $p \geq n$. Finally, it is easily seen that if $k$ is the number of non-simple real zeros, then the total number of real zeros, counting multiplicity, must be at least $2k + p$. However the total number of real zeros, again counting multiplicity, cannot exceed $\deg \Psi_n = n + D-2$. Hence $2k + n \leq 2k+p \leq n+D-2$. Hence $k \leq D/2 - 1$.
\end{proof}

\subsection{The multiple-orthogonality characterisation}

\begin{lemma}\label{linearindep} Define the linear functionals
\begin{align*}
\Xi_k : \mathscr{P}_{n +D-2}^{\mathbb{C}} \longrightarrow \mathbb{C} \, , & &k \in [\![ 0, n ]\!] \\ 
\Theta_j : \mathscr{P}_{n +D-2}^{\mathbb{C}} \longrightarrow \mathbb{C} \, , & & j \in [\![ 1, D-1 ]\!] \setminus \{ D/2 \}
\end{align*}
by
\begin{align*}
\Xi_k [Q] &= \int_{\mathbb{R}} Q(x) x^k \mathrm{e}^{-V(x)}\, \mathrm{d}x \\
\Theta_j [Q] &= -\int_0^{\infty_{2j}} Y_{2j}(z) \mathrm{e}^{-\frac{1}{2}V(z)} Q(z) \, \frac{\mathrm{d}z}{2\pi \mathrm{i}} + \beta_j \int_0^{\infty_0} Y_{0}(z) \mathrm{e}^{-\frac{1}{2}V(z)} Q(z) \, \frac{\mathrm{d}z}{2\pi \mathrm{i}} \\
&\quad + \kappa_j \int_0^{\infty_D} Y_{D}(z) \mathrm{e}^{-\frac{1}{2}V(z)} Q(z) \, \frac{\mathrm{d}z}{2\pi \mathrm{i}} \, .
\end{align*}
Then $\{ \Xi_k \}_{k \in [\![0,n]\!]} \cup \{ \Theta_j \}_{j \in [\![1, D-1]\!]\setminus\{D/2\}}$
is a basis of the dual space $(\mathscr{P}_{n+D-2}^{\mathbb{C}})^\ast$.\end{lemma}
\begin{proof}
Since we have $n+D-1$ linear functionals in total, by the rank-nullity theorem it is sufficient to show that
$$\bigcap_{k=0}^n \ker \Xi_k \cap \bigcap_{\substack{j=1 \\ j \neq D/2}}^{D-1} \ker \Theta_j = 0 \, .$$
Thus let $Q \in \ker \Xi_k$ for all $k \in [\![ 0, n ]\!]$ and $Q \in \ker \Theta_j$ for all $ j \in [\![ 1, D-1]\!] \setminus \{ D/2 \}$. Hence by \cref{polynomialcondition}, there exists a polynomial $P \in \mathscr{P}_n^{\mathbb{C}}$ such that
$$Q = \lambda V^\prime P^\prime - \lambda P^{\prime\prime} + P \, .$$
If we now use that $\Xi_k [Q] = 0$, and integrate by parts, we find that
\begin{align*}
\int_{\mathbb{R}} x^k P(x) \mathrm{e}^{-V(x)}\, \mathrm{d}x + \lambda \int_{\mathbb{R}} \frac{\mathrm{d}}{\mathrm{d}x} (x^k) P^\prime(x) \mathrm{e}^{-V(x)}\, \mathrm{d}x  = 0& & \forall k \in [\![ 0, n ]\!] \, .
\end{align*}
By taking linear combinations of these equations, we can replace $x^k$ with any polynomial of degree $\leq n$. In particular, we can replace $x^k$ with $\overline{P}$. Then
\begin{align*}
0 = \int_{\mathbb{R}} |P(x)|^2 \mathrm{e}^{-V(x)}\, \mathrm{d}x + \lambda \int_{\mathbb{R}} |P^\prime(x)|^2 \mathrm{e}^{-V(x)}\, \mathrm{d}x   \geq \int_{\mathbb{R}} |P(x)|^2 \mathrm{e}^{-V(x)}\, \mathrm{d}x \, .
\end{align*}
Hence $P \equiv 0$ and therefore $Q \equiv 0$.
\end{proof}
\begin{definition}\label{Rdef} Let $n \geq 1$. Define $R_n^{(s)}$, for $s \in [\![ 1 , D-1 ]\!] \setminus \{D/2\}$, to be the unique (complex) polynomial of degree $\leq n + D-3$ such that $\Xi_k[R_n^{(s)} ] = 0$ for all $k \in [\![0, n-1]\!]$ and $\Theta_j[R_n^{(s)} ] = \delta_{sj}$ for all $j \in [\![ 1 , D-1 ]\!] \setminus \{D/2\}$. By \cref{linearindep}, $R_n^{(s)}$ exists and is unique.
\end{definition}
We finally arrive at one of our principal results, namely that $\{ \Psi_n \}_{n \geq 1}$ can be regarded as a kind of multiple-orthogonal polynomial of Type II.
\begin{corollary}[Dual Sobolev orthogonal polynomials as a Type II multiple orthogonal polynomial]\label{dualsobolevMOP} Let $n \geq 1$ be an integer. Consider the following problem. We are asked to find a monic polynomial of degree exactly $n+D-2$ such that for all $k \in [\![0, n-1 ]\!]$
\begin{align*}
\int_{\mathbb{R}} x^k Q(x) \mathrm{e}^{-V(x)}\, \mathrm{d}x = 0
\end{align*}
and for all $j \in [\![ 1, D-1 ]\!] \setminus \{ D/2 \}$
\begin{align*}
\int_0^{\infty_{2j}} Y_{2j}(z) \mathrm{e}^{-\frac{1}{2}V(z)} Q(z) \, \mathrm{d}z - \beta_j \int_0^{\infty_0} Y_{0}(z) \mathrm{e}^{-\frac{1}{2}V(z)} Q(z) \, \mathrm{d}z - \kappa_j \int_0^{\infty_D} Y_{D}(z) \mathrm{e}^{-\frac{1}{2}V(z)} Q(z) \, \mathrm{d}z  = 0 \, .
\end{align*}
Then this problem has the unique solution $Q \equiv \Psi_n$.
\end{corollary}
However our original interest was in the Sobolev orthogonal polynomials $\{ P_n \}_{n \in \mathbb{N}}$ and not their duals.
\begin{corollary}[Sobolev orthogonal polynomials as a Type I multiple orthogonal polynomial]\label{sobolevMOP} Let $n \geq 1$, and consider the following problem. We are asked to find a monic polynomial $P$ of degree exactly $n$ and a vector of constants $\boldsymbol{c} = (c_j)_{j \in [\![1, D-1]\!]\setminus\{D/2\}}$ such that
\begin{align*}
&\int_{\mathbb{R}} x^k P(x) \mathrm{e}^{-V(x)} \, \mathrm{d}x + \\
& \sum_{\substack{j = 1 \\ j \neq D/2}}^{D-1} c_j \Big\{ \int_0^{\infty_{2j}} Y_{2j}(z) \mathrm{e}^{-\frac{1}{2}V(z)} z^k \, \mathrm{d}z - \beta_j \int_0^{\infty_0} Y_{0}(z) \mathrm{e}^{-\frac{1}{2}V(z)} z^k \, \mathrm{d}z - \kappa_j \int_0^{\infty_D} Y_{D}(z) \mathrm{e}^{-\frac{1}{2}V(z)} z^k \, \mathrm{d}z \Big\} = 0
\end{align*}
for all $k \in [\![0, n+D-3]\!]$. Then this problem has a unique solution $(P, \boldsymbol{c})$, where $P = P_n$.
\end{corollary}
\begin{proof}
By taking linear combinations, we can replace $x^k$ and $z^k$ by any polynomial $Q$ of degree $\leq n+D-3$. Let us expand $P(x) = b_0 + b_1 x + \dots + b_{n-1}x^{n-1} + x^n$. Then, rearranging, we are asked to find $(b_i)_{i \in [\![ 0, n-1 ]\!]}$ and $(c_j)_{j \in [\![1, D-1]\!]\setminus\{D/2\}}$ such that
$$\sum_{i=0}^{n-1} b_i \Xi_i[Q] + \sum_{\substack{j=1 \\ j \neq D/2}}^{D-1} c_j \Theta_j [Q] = - \Xi_n[Q] \, .$$
Since this must hold for all $Q$ of degree $\leq n+D-3$, this must be an identity at the level of linear functionals in $(\mathscr{P}_{n+D-3}^{\mathbb{C}})^\ast$. But by \cref{linearindep}, the linear functional $- \Xi_n|_{\mathscr{P}_{n+D-3}^{\mathbb{C}}}$ has a unique expansion in the basis $\{ \Xi_k \}_{k \in [\![0,n-1]\!]} \cup \{ \Theta_j \}_{j \in [\![1, D-1]\!]\setminus\{D/2\}}$. This establishes in a stroke both existence and uniqueness.

Finally, let us in particular take $Q$ to be $\Psi_0, \Psi_1, \dots , \Psi_{n-1}$. Then by \cref{polynomialcondition}, we have
\begin{align*}
\int_{\mathbb{R}} \Psi_k(x) P(x) \mathrm{e}^{-V(x)} \, \mathrm{d}x = 0 & & \forall k \in [\![ 0, n-1 ]\!] \, .
\end{align*}
This implies that $\langle P_k, P \rangle_S = 0$ for all $k \in [\![ 0, n-1 ]\!]$. Since $P$ is monic of degree $n$, we must have $P = P_n$. Next, again by taking linear combinations, let us replace $x^k$ and $z^k$ with $R^{(s)}_n$ for $s \in [\![ 1, D-1 ]\!] \setminus \{ D/2\}$. This gives
$$\int_{\mathbb{R}} R^{(s)}_n(x) P_n(x) \mathrm{e}^{-V(x)}\, \mathrm{d}x - 2\pi \mathrm{i} c_s = 0$$
which determines $c_s$.
\end{proof}

\subsection{The Riemann–Hilbert problems}

It is well-known that multiple orthogonal polynomials of both Type I and Type II admit a Riemann--Hilbert representation \cite{vanassche_etal_2001}. Let
\begin{align*}
\Sigma_k &\overset{\mathrm{def}}{=} \mathrm{e}^{2\pi \mathrm{i}k/D}[0,+\infty)\, , & & k \in [\![ 0, D-1 ]\!] \\
\Sigma &\overset{\mathrm{def}}{=} \bigcup_{k=0}^{D-1} \Sigma_k \, .
\end{align*}
All these contours have orientation away from the origin except for $\Sigma_{D/2} = (-\infty, 0]$ which has orientation towards the origin. Note that this means that $\mathbb{R} = \Sigma_{D/2} \cup \Sigma_0$ with standard left to right orientation. The contours are depicted for the case $D = 8$ in Figure \ref{fig:contours}.

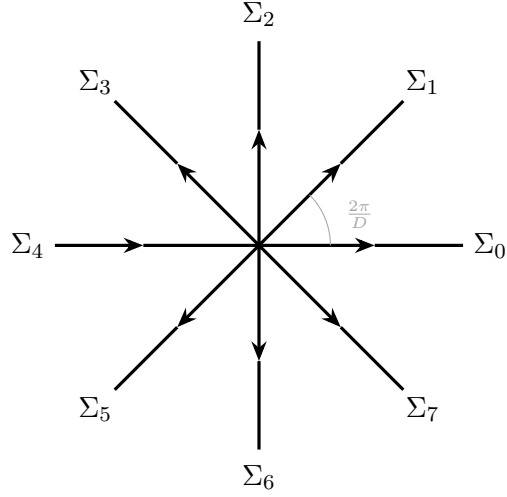
\begin{figure}[tbp]
\centering
\def\Dd{8}         
\def\Rout{3}       
\def\Rmid{1.7}     
\def\Rlab{3.4}     
\begin{tikzpicture}[
    scale=0.9,
    every node/.style={font=\small},
    tip/.style={-{Stealth[length=2.6mm,width=2mm]}},
    pathline/.style={very thick}
  ]
  \foreach \k in {0,1,2,3,5,6,7}{
    \draw[pathline,tip] (0,0) -- ({360*\k/\Dd}:\Rmid);
    \draw[pathline]     ({360*\k/\Dd}:\Rmid) -- ({360*\k/\Dd}:\Rout);
  }
  \draw[pathline,tip] (180:\Rout) -- (180:\Rmid);
  \draw[pathline]     (180:\Rmid) -- (0,0);
  \draw[gray!70,thin] (0:1.05) arc (0:{360/\Dd}:1.05);
  \node[gray!70,font=\scriptsize,anchor=west] at ({180/\Dd}:1.2)
        {$\tfrac{2\pi}{D}$};
  \fill (0,0) circle (1.7pt);
  \foreach \k in {0,...,7}{
    \node at ({360*\k/\Dd}:\Rlab) {$\Sigma_{\k}$};
  }
\end{tikzpicture}
\caption{The contours $\Sigma_0,\dots,\Sigma_{D-1}$ and their orientations,
drawn for $D=8$.  Note the inward orientation of $\Sigma_{D/2}$.}
\label{fig:contours}
\end{figure}

Define, for $j \in [\![ 1, D-1 ]\!] \setminus \{ D/2 \}$
\begin{equation}
\begin{split}
&\omega_j : \Sigma \longrightarrow \mathbb{C} \\
&\omega_j(z) = Y_{2j}(z) \mathrm{e}^{-\frac{1}{2}V(z)}\chi_{\Sigma_j}(z) - \beta_j  Y_0(z) \mathrm{e}^{-\frac{1}{2}V(z)} \chi_{\Sigma_0}(z) + \kappa_j  Y_D(z) \mathrm{e}^{-\frac{1}{2}V(z)} \chi_{\Sigma_{D/2}}(z) \, .
\end{split}
\end{equation}
\begin{rhp}\label{RHP1} Let $n \geq 1$ and $V$ be a polynomial of even degree $D\geq 4$ with positive leading coefficient. We look for a $D \times D$ matrix valued function
$$A_n : \mathbb{C}\setminus \Sigma \longrightarrow \mathcal{M}_D(\mathbb{C})$$
such that the following is true.
\begin{enumerate}[label=(\arabic*)]
\item $A_n$ is analytic on $\mathbb{C}\setminus \Sigma$ (i.e. its matrix elements are analytic).
\item $A_n$ has continuous boundary values up to $\Sigma \setminus\{0 \}$. That is, given any $x \in \Sigma \setminus\{0 \}$, as $z \in \mathbb{C}\setminus \Sigma$ tends to $x$, from either the left ($+$) or the right ($-$) side, non-tangentially, the limit exists. These limits are denoted
\begin{align*}
A_n^+ : \Sigma \setminus\{0 \} \longrightarrow \mathcal{M}_D(\mathbb{C}), & & A_n^- : \Sigma \setminus\{0 \} \longrightarrow \mathcal{M}_D(\mathbb{C}) 
\end{align*}
and are continuous functions. \enquote{Left} and \enquote{right} mean relative to the orientation of the contour. If we define the row vector
\begin{align}\label{omegavector}
\boldsymbol{\omega}(x) \overset{\mathrm{def}}{=} \big( \omega_1(x), \dots, \omega_{D/2-1}(x),
\omega_{D/2+1}(x), \dots, \omega_{D-1}(x) \big)
\end{align}
then the boundary values are related by the \enquote{jump matrix}
\begin{align*}
A^+_n(x) = A^-_n(x) \begin{pmatrix}
1 & \mathrm{e}^{-V(x)} \chi_{\mathbb{R}}(x) & \boldsymbol{\omega}(x) \\
0 & 1 & \mathbf{0} \\
\mathbf{0} & \mathbf{0} & \mathbb{I}_{D-2}
\end{pmatrix}\, , & & x \in \Sigma\setminus \{ 0 \} \, .
\end{align*}
\item $A_n$ is bounded in a neighbourhood of $0$.
\item We have the asymptotic normalisation
\begin{align*}
A_n(z) = \bigl(\mathbb{I} + \mathcal{O}(z^{-1})\bigr)
\operatorname{diag}\bigl(z^{n+D-2},\, z^{-n},\, z^{-1},\, \dots,\, z^{-1}\bigr)\, , & & z \to \infty \, .
\end{align*}
\end{enumerate}
\end{rhp}
\begin{lemma}\label{uniqueness} If a solution to Riemann--Hilbert problem \ref{RHP1} exists, then it is unique.
\end{lemma}
\begin{proof}
This proof is classical but we repeat it for completeness. Let $A_n$ be a solution of RHP \ref{RHP1}. We begin by noting that $z \mapsto \det A_n(z)$ is analytic on $\mathbb{C}\setminus \Sigma$, has continuous boundary values up to $\Sigma \setminus \{ 0 \}$, and
\begin{align*}
\det A_n^+(x) = \det A_n^-(x) & & \forall x \in \Sigma \setminus \{ 0 \}
\end{align*}
Hence by Morera's theorem $z \mapsto \det A_n(z)$ may be extended to an analytic function on $\mathbb{C}\setminus \{ 0 \}$. But since $A_n$ is bounded at $0$, the singularity at $0$ is removable, hence $\det A_n$ may be extended to an entire function. By the asymptotic condition, $\det A_n(z) = 1 + \mathcal{O}(z^{-1})$. $\det A_n$ is therefore entire and bounded, and so constant by Liouville's theorem. By the asymptotic condition $\det A_n \equiv 1$. From this we deduce that $A_n^{-1}$ is analytic on $\mathbb{C}\setminus \Sigma$ and has continuous boundary values $(A_n^{\pm})^{-1}$, and is bounded at $0$.

Next, let $\widetilde{A}_n$ be any other solution to RHP \ref{RHP1} and consider the quantity $C_n(z) := \widetilde{A}_n(z) A_n(z)^{-1}$. This function is analytic on $\mathbb{C}\setminus \Sigma$ with continuous boundary values up to $\Sigma \setminus \{ 0 \}$, and from the jump condition
\begin{align*}
C_n^+(x) = C_n^-(x) \, , & & \forall x \in \Sigma \setminus \{ 0 \} \, .
\end{align*}
Hence by Morera's theorem $C_n$ may be extended to an analytic function on $\mathbb{C}\setminus \{ 0 \}$. Since $C_n$ is bounded at $0$, this singularity is removable, and hence $C_n$ may be extended to an entire function. By the asymptotic condition
\begin{align*}
C_n(z) = \mathbb{I}+\mathcal{O}(z^{-1}) \, , & & z \to \infty
\end{align*}
and so by Liouville's theorem $C_n \equiv \mathbb{I}$.
\end{proof}
Let us now construct the solution of RHP \ref{RHP1}. From the jump condition, we have
\begin{align*}
(A_n^+)_{i1} = (A_n^-)_{i1} & & \forall  i \in [\![ 1, D ]\!] \, .
\end{align*}
Furthermore, by boundedness at $0$, we see that the singularity at $0$ is removable, hence $(A_n)_{i1}$ are entire functions. From the asymptotic condition
\begin{align*}
(A_n)_{i1}(z) &= \delta_{i1} z^{n+D-2} + \mathcal{O}(z^{n+D-3}) & &  i \in [\![ 1, D]\!] \, , \, z \to \infty \, . \\
\end{align*}
Hence by Liouville's theorem $(A_n)_{11}$ is a monic polynomial of degree exactly $n+D-2$, and $(A_n)_{i1}$, for $i \geq 2$, are polynomials of degree at most $n+D-3$.

Moving onto the second column, the jump condition gives
\begin{align*}
(A_n^+)_{i 2 }(x) = \mathrm{e}^{-V(x)} \chi_{\mathbb{R}}(x) (A_n)_{i1}(x) +   (A_n^-)_{i 2 }(x) \, , & & i \in [\![1, D]\!]\, , x \in \Sigma \setminus \{ 0 \} \, .
\end{align*}
Then from the Plemelj formula we have $(A_n)_{i 2 }(z) = C_{\mathbb{R}}(\mathrm{e}^{-V}(A_n)_{i1})(z)$. If we now expand this as $z \to \infty$ we find
\begin{align*}
C_{\mathbb{R}}(\mathrm{e}^{-V}(A_n)_{i1})(z) = - \frac{1}{2\pi \mathrm{i}} \sum_{k=0}^{n-1} z^{-k-1} \int_{\mathbb{R}} x^k \mathrm{e}^{-V(x)} (A_n)_{i1}(x) \, \mathrm{d}x + \mathcal{O}(z^{-n-1}) \, . 
\end{align*}
From the asymptotic condition we have $(A_n)_{i 2 }(z) = \delta_{i2} z^{-n} + \mathcal{O}(z^{-n-1})$. Hence 
\begin{align}\label{orthgonality}
\int_{\mathbb{R}} x^k \mathrm{e}^{-V(x)} (A_n)_{i1}(x) \, \mathrm{d}x = - 2\pi \mathrm{i} \delta_{i2}\delta_{k,n-1} & & \forall k \in [\![0, n-1 ]\!] \, , \, i \in [\![1, D ]\!] \, .
\end{align}
Considering the jump condition for the third to $D$th columns, we find that
\begin{align}\label{psiconditions}
\int_{\Sigma} \omega_j(x) (A_n)_{11}(x) \, \mathrm{d}x = \int_{\Sigma} \omega_j(x) (A_n)_{21}(x) \, \mathrm{d}x = 0  & & \forall j \in [\![ 1, D-1]\!] \setminus \{ D/2\} \, .
\end{align}
Combining \eqref{orthgonality} and \eqref{psiconditions}, we see by  \cref{dualsobolevMOP} that
\begin{align*}
(A_{n})_{11} &= \Psi_n \\
(A_{n})_{21} &= \begin{cases}  - \frac{2\pi \mathrm{i} \lambda D (n-1) \gamma}{h_{n-1}} \Psi_{n-1} & n \geq 2 \\
- \frac{2\pi \mathrm{i}}{h_0} & n = 1 \, .
\end{cases}
\end{align*}
Hence the first row of $A_n$ is
$$ \Big( \Psi_n, C_{\mathbb{R}}(\mathrm{e}^{-V} \Psi_n), C_{\Sigma}(\boldsymbol{\omega} \Psi_n)  \Big) $$
and the second row is, for $n \geq 2$
$$-\frac{2\pi \mathrm{i} \lambda D (n-1) \gamma}{h_{n-1}} \Big( \Psi_{n-1}, C_{\mathbb{R}}(\mathrm{e}^{-V} \Psi_{n-1}), C_{\Sigma}(\boldsymbol{\omega} \Psi_{n-1}) \Big)$$
whilst for $n = 1$ it is $-\frac{2\pi \mathrm{i} }{h_{0}} \Big( 1 , C_{\mathbb{R}}(\mathrm{e}^{-V} ), C_{\Sigma}(\boldsymbol{\omega} ) \Big)$. By a similar argument we can compute the remaining rows. Recall the definition of $R_n^{(s)}$ in Definition \ref{Rdef}, and define the column vector
\begin{align}\label{Rvector}
\boldsymbol{R}_n(x) \overset{\mathrm{def}}{=} \big( R_n^{(1)}(x), \dots, R_n^{(D/2-1)}(x),
R_n^{(D/2+1)}(x), \dots, R_n^{(D-1)}(x) \big)^\mathsf{T} \, .
\end{align}
From the definition of $R_n^{(s)}$, we have $C_{\mathbb{R}}(\mathrm{e}^{-V}\boldsymbol{R}_n)(z) = \mathcal{O}(z^{-n-1})$ and $C_{\Sigma}(\boldsymbol{R}_n \boldsymbol{\omega})(z) = z^{-1} \mathbb{I}_{D-2}+ \mathcal{O}(z^{-2})$.
\begin{remark}
The fact that the non-tangential limits exist up to $\Sigma\setminus\{ 0 \}$ for all the aforementioned Cauchy  transforms follows from the fact that the integrand is analytic, and hence we may deform the contour.
\end{remark}
The RHP \ref{RHP1} requires the solution to be bounded in a neighbourhood of $0$, hence we claim that all the aforementioned Cauchy transforms are bounded in a neighbourhood of $0$.
\begin{lemma} Let $f$ be an entire function such that $\int_{\Sigma} |f(z) \omega_j(z) | |\mathrm{d}z| < + \infty$ for all $j \in [\![1, D-1]\!]\setminus \{ D/2\}$. Then for all $j \in [\![1, D-1]\!]\setminus \{ D/2\}$, the Cauchy transform $C_{\Sigma}(f \omega_j)(\zeta)$ is bounded in a neighbourhood of $\zeta = 0$.
\end{lemma}
\begin{proof}
Let us write $C_{\Sigma}(f \omega_j)(\zeta) = C_{\Sigma \cap D(0,1)}(f \omega_j)(\zeta)+ C_{\Sigma \setminus D(0,1)}(f \omega_j)(\zeta)$. $C_{\Sigma \setminus D(0,1)}(f \omega_j)(\zeta)$ is clearly bounded in a neighbourhood of $\zeta = 0$, hence this leaves $C_{\Sigma \cap D(0,1)}(f \omega_j)(\zeta)$. Explicitly, we have
\begin{align*}
C_{\Sigma \cap D(0,1)}(f \omega_j)(\zeta) &= \frac{1}{2\pi \mathrm{i}}\int_0^{\mathrm{e}^{2\pi \mathrm{i}j /D}}Y_{2j}(z) \mathrm{e}^{-\frac{1}{2}V(z) } f(z) \frac{1}{z-\zeta} \, \mathrm{d}z - \frac{1}{2\pi \mathrm{i}} \beta_j \int_0^{1}Y_{0}(z) \mathrm{e}^{-\frac{1}{2}V(z) } f(z) \frac{1}{z-\zeta} \, \mathrm{d}z \\
&\quad - \frac{1}{2\pi \mathrm{i}} \kappa_j \int_0^{-1} Y_{D}(z) \mathrm{e}^{-\frac{1}{2}V(z)} f(z) \frac{1}{z-\zeta} \, \mathrm{d}z \, .
\end{align*}
Let
\begin{align*}
g_1(z) &= Y_{2j}(z) \mathrm{e}^{-\frac{1}{2}V(z)} f(z) - Y_{2j}(0) \mathrm{e}^{-\frac{1}{2}V(0)} f(0) \\
g_2(z) &= \beta_j Y_{0}(z) \mathrm{e}^{-\frac{1}{2} V(z)} f(z) - \beta_j Y_{0}(0) \mathrm{e}^{-\frac{1}{2}V(0)} f(0) \\
g_3(z) &= \kappa_j Y_{D}(z) \mathrm{e}^{-\frac{1}{2}V(z)} f(z) - \kappa_j Y_{D}(0) \mathrm{e}^{-\frac{1}{2}V(0)} f(0) \, .
\end{align*}
Then by \eqref{identity}, we see that
\begin{align*}
C_{\Sigma \cap D(0,1)}(f \omega_j)(\zeta) &= \frac{1}{2\pi \mathrm{i}}\int_0^{\mathrm{e}^{2\pi \mathrm{i}j /D}}g_1(z) \frac{1}{z-\zeta} \, \mathrm{d}z - \frac{1}{2\pi \mathrm{i}}  \int_0^{1} g_2(z) \frac{1}{z-\zeta} \, \mathrm{d}z \\
&\quad - \frac{1}{2\pi \mathrm{i}} \int_0^{-1} g_3(z) \frac{1}{z-\zeta} \, \mathrm{d}z 
+ \frac{1}{2\pi \mathrm{i}} Y_{2j}(0) \mathrm{e}^{-\frac{1}{2}V(0)} f(0) \log (\zeta - \mathrm{e}^{2\pi \mathrm{i}j /D}) \\
&\quad - \frac{1}{2\pi \mathrm{i}} \beta_j Y_{0}(0) \mathrm{e}^{-\frac{1}{2}V(0)} f(0) \log (\zeta - 1) - \frac{1}{2\pi \mathrm{i}} \kappa_j Y_{D}(0) \mathrm{e}^{-\frac{1}{2}V(0)} f(0) \log (\zeta + 1) \, .
\end{align*}
By rescaling variables, all three integrals can be put in the form $\int_0^1 g(z) \frac{1}{z-\zeta} \, \mathrm{d}z$ for some entire function $g$ such that $g(0) = 0$. We may rewrite this as $\int_0^1 \frac{g(z) - g(\zeta)}{z-\zeta} \, \mathrm{d}z + g(\zeta) \log \frac{\zeta-1}{\zeta}$. The first integral term is bounded on compact sets, and the second term is bounded in a neighbourhood of $\zeta = 0$ because $g(0) = 0$.
\end{proof}
\begin{remark} The fact that $A_n$ is bounded at $0$ is made possible by fact that the jump matrix obeys the appropriate cyclicity condition around $0$ (see Definition 2.55 in \cite{trogdon2015riemann}), which is a self-intersection point of $\Sigma$.
\end{remark}

From this we conclude the following.
\begin{theorem}\label{RHP1solution}
The unique solution of the Riemann--Hilbert problem \ref{RHP1} is
\begin{align*}
A_n = \begin{pmatrix}
\Psi_n & C_{\mathbb{R}}\big(\mathrm{e}^{-V} \Psi_n\big) & C_{\Sigma}\big(\Psi_n \boldsymbol{\omega}\big) \\[2pt]
- \frac{2\pi \mathrm{i}\lambda}{\nu_{n-1} h_{n-1}}\Psi_{n-1} & - \frac{2\pi \mathrm{i}\lambda}{\nu_{n-1} h_{n-1}} C_{\mathbb{R}}\big(\mathrm{e}^{-V} \Psi_{n-1}\big) &  - \frac{2\pi \mathrm{i}\lambda}{\nu_{n-1} h_{n-1}} C_{\Sigma}\big(\Psi_{n-1} \boldsymbol{\omega}\big) \\[2pt]
\boldsymbol{R}_n & C_{\mathbb{R}}\big(\mathrm{e}^{-V} \boldsymbol{R}_n\big) & C_{\Sigma}\big(\boldsymbol{R}_n \boldsymbol{\omega}\big)
\end{pmatrix} \, .
\end{align*}
\end{theorem}
\begin{remark} Let us remark that RHP \ref{RHP1} has no solution for $n = 0$. This is because $\Psi_0$, rather than being monic of degree $D-2$, is monic of degree $0$ (namely, the constant polynomial $1$). It is easily seen that there is no polynomial $P$ which can make the polynomial $P - \lambda P^{\prime\prime} + \lambda V^\prime P^\prime$ have degree exactly $D-2$. For similar reasons, the polynomials $R^{(s)}_n$ are undefined for $n = 0$ hence we must require $n \geq 1$ in \cref{Rdef}.
\end{remark}
We have already established in the proof of \cref{uniqueness} that $\det A_n \equiv 1$. Given this, it is natural to consider the \enquote{dual} Riemann--Hilbert problem $\widehat{A_n} = A_n^{-\mathsf{T}}$ (inverse transpose).
\begin{rhp}\label{RHP2} Let $n \geq 1$ and let $V$ be a polynomial of even degree $D\geq 4$ with positive leading coefficient. We look for a $D \times D$ matrix valued function
$$\widehat{A_n} : \mathbb{C}\setminus \Sigma \longrightarrow \mathcal{M}_D(\mathbb{C})$$
such that the following is true.
\begin{enumerate}[label=(\arabic*)]
\item $\widehat{A_n}$ is analytic on $\mathbb{C}\setminus \Sigma$ (i.e. its matrix elements are analytic).
\item $\widehat{A_n}$ has continuous boundary values up to $\Sigma \setminus\{0 \}$. That is, given any $x \in \Sigma \setminus\{0 \}$, as $z \in \mathbb{C}\setminus \Sigma$ tends to $x$, from either the left ($+$) or the right ($-$) side, non-tangentially, the limit exists. These limits are denoted
\begin{align*}
\widehat{A_n}^+ : \Sigma \setminus\{0 \} \longrightarrow \mathcal{M}_D(\mathbb{C}), & & \widehat{A_n}^- : \Sigma \setminus\{0 \} \longrightarrow \mathcal{M}_D(\mathbb{C})
\end{align*}
and are continuous functions. \enquote{Left} and \enquote{right} mean relative to the orientation of the contour. The boundary values are related by the \enquote{jump matrix}
\begin{align*}
\widehat{A_n}^+(x) = \widehat{A_n}^-(x) \begin{pmatrix}
1 & 0 & \mathbf{0} \\
-\mathrm{e}^{-V(x)} \chi_{\mathbb{R}}(x) & 1 & \mathbf{0} \\
-\boldsymbol{\omega}(x)^{\mathsf{T}} & \mathbf{0} & \mathbb{I}_{D-2}
\end{pmatrix}\, , & & x \in \Sigma\setminus \{ 0 \} \, ,
\end{align*}
which is the inverse transpose of the jump matrix of RHP \ref{RHP1}.
\item $\widehat{A_n}$ is bounded in a neighbourhood of $0$.
\item We have the asymptotic normalisation
\begin{align*}
\widehat{A_n}(z) = \big(\mathbb{I} + \mathcal{O}(z^{-1})\big) \operatorname{diag}\big( z^{-n-D+2},\, z^{n},\, z,\, \dots,\, z \big) \, , & & z \to \infty \, .
\end{align*}
\end{enumerate}
\end{rhp}
Existence of a solution to RHP \ref{RHP2} follows from \cref{RHP1solution}, since we have already constructed $A_n$ and shown that $\det A_n \equiv 1$. Uniqueness follows by the same argument as \cref{uniqueness}. Let us compute the $(2,2)$ matrix element of $\widehat{A_n}$. It is possible to compute the other matrix elements, however this one will turn out to be the most interesting to us.

From the jump condition we see that
\begin{align*}
(\widehat{A_n}^+)_{ij}(x) = (\widehat{A_n}^-)_{ij}(x) & & i \in [\![1, D]\!]\, , \, j \in [\![2, D]\!] \, .
\end{align*}
From boundedness at $0$ we conclude, by Morera's theorem, that $(\widehat{A_n})_{ij}$ are entire functions, for $i \in [\![1, D]\!]$ and $j \in [\![2, D]\!]$. Next, from the normalisation condition,
$$(\widehat{A_n})_{i2}(z) = \delta_{i2} z^{n} + \mathcal{O}(z^{n-1}) \, .$$
Hence by Liouville's theorem $(\widehat{A_n})_{22}$ is a monic polynomial of degree exactly $n$, whilst $(\widehat{A_n})_{i2}$ for $i \neq 2$ is a polynomial of degree at most $n-1$. By a similar argument, one sees that
\begin{align*}
(\widehat{A_{n}})_{ij}(z) = \delta_{ij} z + c_{ij} \, , & & i \in [\![1, D]\!] \, , \, j \in [\![ 3, D ]\!]
\end{align*}
for some collection of constants $c_{ij}$. Let us gather the constants appearing in the second row into the column vector
\begin{align}\label{cvector}
\boldsymbol{c}_2 \overset{\mathrm{def}}{=} \big( c_{2 3}, \, \dots , \, c_{2 D} \big)^{\mathsf{T}} \in \mathbb{C}^{D-2} \, ,
\end{align}
whose entries are ordered so as to match those of $\boldsymbol{\omega}$ in \eqref{omegavector}. Next, let us consider the $(2,1)$ matrix element. We see that
\begin{align*}
(\widehat{A_n}^+)_{21}(x) = (\widehat{A_n}^-)_{21}(x) - \mathrm{e}^{-V(x)}\chi_{\mathbb{R}}(x) (\widehat{A_n})_{22}(x) - \boldsymbol{\omega}(x) \boldsymbol{c}_2 \, , & & x \in \Sigma \setminus \{ 0 \} \, .
\end{align*}
By the Plemelj formula we then have
\begin{align*}
(\widehat{A_n})_{21}(z) = - C_{\mathbb{R}}\big(\mathrm{e}^{-V} (\widehat{A_n})_{22}\big)(z) - C_{\Sigma}(\boldsymbol{\omega})(z) \, \boldsymbol{c}_2 \, .
\end{align*}
From the scaling we have $(\widehat{A_n})_{21}(z) = \mathcal{O}(z^{-n-D+1})$, which implies that
\begin{align*}
\int_{\mathbb{R}} x^k \, \mathrm{e}^{-V(x)} (\widehat{A_n})_{22} (x) \, \mathrm{d}x + \bigg( \int_{\Sigma} x^k \, \boldsymbol{\omega}(x) \, \mathrm{d}x \bigg) \boldsymbol{c}_2 = 0 \, , & & \forall k \in [\![0, n+D-3]\!] \, .
\end{align*}
Hence by \cref{sobolevMOP}, $(\widehat{A_n})_{22} = P_n$, our $n$th monic Sobolev orthogonal polynomial. We record this as follows.
\begin{theorem}\label{RHP2solution}
The unique solution $\widehat{A_n} = A_n^{-\mathsf{T}}$ of the Riemann--Hilbert problem \ref{RHP2} satisfies
\begin{align*}
(\widehat{A_n})_{22} = (A_n^{-1})_{22} = P_n \, ,
\end{align*}
the $n$th monic Sobolev orthogonal polynomial.
\end{theorem}
\begin{lemma}[Symmetry of the RHP] Let $\mathsf{P}$ be the permutation matrix on $\mathbb{R}^D$ which fixes $\mathsf{e}_1 = \mathsf{P} \mathsf{e}_1$ and  $\mathsf{e}_2 = \mathsf{P} \mathsf{e}_2$, and reverses the order of the remaining basis vectors, $\mathsf{P}\mathsf{e}_j = \mathsf{e}_{D-j+3}$ for $j \in [\![ 3, D ]\!]$. Let $\mathsf{D} = \operatorname{diag} (1,-1, - \mathsf{i}, \dots, - \mathsf{i})$. Define $\overline{A_n}(z) := \overline{A_n(\overline{z})}$. Then \cref{RHP1} possesses the symmetry $$A_n(z)  = \mathsf{P}\mathsf{D} \overline{A_n}(z) \mathsf{D}^{-1} \mathsf{P} \, .$$
\end{lemma}
\begin{proof}
We show that the right hand side also solves \cref{RHP1} and so by uniqueness (\cref{uniqueness}) we deduce the equality. To do this, let us use $J(z)$ to denote the jump matrix in \cref{RHP1}, i.e.
$$J(z) := \begin{pmatrix}
1 & \mathrm{e}^{-V(x)} \chi_{\mathbb{R}}(x) & \boldsymbol{\omega}(x) \\
0 & 1 & \mathbf{0} \\
\mathbf{0} & \mathbf{0} & \mathbb{I}_{D-2}
\end{pmatrix} \, .$$
 Then $J$ possesses the symmetry $J(z) = \mathsf{P}\mathsf{D} \overline{J}(z)^{-1} \mathsf{D}^{-1}\mathsf{P}$, which may be verified by direct calculation. Moving forward, let $\widetilde{A}_n := \mathsf{P}\mathsf{D}\overline{A_n}(z) \mathsf{D}^{-1}\mathsf{P}$ and we will show that $\widetilde{A}_n = A_n$. We begin by remarking that $\widetilde{A}_n$ trivially satisfies all the requirements of \cref{RHP1} except for the jump condition which is slightly less trivial. Complex conjugation reverses which side of $\Sigma \setminus \{ 0 \}$ is the $+$ side versus the $-$ side. Hence, for $z \in \Sigma \setminus \{ 0 \}$
 \begin{align*}
 \widetilde{A}_n^+(z) = \mathsf{P}\mathsf{D}\overline{A_n^-(\overline{z})} \mathsf{D}^{-1}\mathsf{P} = \mathsf{P}\mathsf{D}\overline{A_n^+(\overline{z})} \overline{J}(z)^{-1} \mathsf{D}^{-1}\mathsf{P} = \widetilde{A}_n^-(z) \underbrace{\mathsf{P}\mathsf{D} \overline{J}(z)^{-1} \mathsf{D}^{-1}\mathsf{P} }_{=J(z)} \, .
 \end{align*}
 Hence $\widetilde{A}_n$ also solves \cref{RHP1}, and so by uniqueness $\widetilde{A}_n = A_n$.
\end{proof}

\subsection{Christoffel--Darboux-type formula}

It is well known that orthogonal polynomials obey a three term recurrence relation. A similar relation holds at the level of the Riemann--Hilbert problems \ref{RHP1} and \ref{RHP2}.
\begin{proposition}[$D+1$-term recurrence relation]\label{recurrence} Let $n \geq 1$ and let $A_n^{(1)}$ denote the next-to-leading term in the $z \to \infty$ expansion of $A_n(z)$, i.e.
\begin{align*}
A_n(z) &= \big(\mathbb{I} + A_n^{(1)} z^{-1}+ \mathcal{O}(z^{-2})\big) \operatorname{diag}\big( z^{n+D-2},\, z^{-n},\, z^{-1},\, \dots,\, z^{-1} \big)  & & z\to \infty \, .
\end{align*}
Let $\mathsf{E}^{ij}$ denote the $D\times D$ matrix with all matrix elements $0$ except for the $(i,j)$ element, which is $1$. Then we have the recursion relations
\begin{align*}
A_{n+1}(z) &= \underbrace{\Big( z \mathsf{E}^{11} + A_{n+1}^{(1)} \mathsf{E}^{11} - \mathsf{E}^{11} A_{n}^{(1)} + \mathbb{I} - \mathsf{E}^{11} - \mathsf{E}^{22} \Big)}_{\Delta_n(z)} A_n(z) \\
\widehat{A_{n+1}}(z) &= \underbrace{\Big( z \mathsf{E}^{22} - \big(A_{n+1}^{(1)}\big)^{\mathsf{T}} \mathsf{E}^{22} + \mathsf{E}^{22} \big(A_{n}^{(1)}\big)^{\mathsf{T}} + \mathbb{I} - \mathsf{E}^{11} - \mathsf{E}^{22} \Big)}_{\widehat{\Delta_n}(z)} \widehat{A_n}(z) \, .
\end{align*}
Finally, we have $\Delta_n(z)^{-1} \Delta_n(w) = (z-w)  \big(A^{(1)}_{n+1}\big)_{21} \mathsf{E}^{21} + \mathbb{I}$.
\end{proposition}
\begin{proof}
Define $\Delta_n(z) = A_{n+1}(z) A_n(z)^{-1}$. We have already established in the proof of \cref{uniqueness} that $\det A_n \equiv 1$, and hence $\Delta_n$ is analytic on $\mathbb{C}\setminus \Sigma$, with continuous left and right boundary values up to $\Sigma \setminus \{ 0 \}$, and is bounded in a neighbourhood of $0$. The jump matrix in RHP \ref{RHP1} does not depend on $n$, and so $\Delta_n^+(x) = \Delta_n^-(x)$. Hence $\Delta_n$ has an analytic extension to the whole of $\mathbb{C}$. Let us then expand $\Delta_n(z)$ as $z \to \infty$,
\begin{align*}
\Delta_n(z) &= \big(\mathbb{I} + A_{n+1}^{(1)} z^{-1}  + \mathcal{O}(z^{-2})\big) \operatorname{diag}\big( z,\, z^{-1},\, 1,\, \dots,\, 1 \big) \big(\mathbb{I} - A_{n}^{(1)} z^{-1}  + \mathcal{O}(z^{-2})\big)  \\
&= z \mathsf{E}^{11} + A_{n+1}^{(1)} \mathsf{E}^{11} - \mathsf{E}^{11} A_{n}^{(1)} + \mathbb{I} - \mathsf{E}^{11} - \mathsf{E}^{22} + \mathcal{O}(z^{-1}) \, .
\end{align*}
However, because $\Delta_n$ is entire and scales like a polynomial, Liouville's theorem asserts it is actually a polynomial. Hence the error term $\mathcal{O}(z^{-1})$ must be identically $0$. A similar argument holds for $\widehat{\Delta_n}(z) = \widehat{A_{n+1}}(z) \widehat{A_n}(z)^{-1}$, upon observing that the next-to-leading term of $\widehat{A_n}(z)$ is $- \big(A_{n}^{(1)}\big)^{\mathsf{T}}$.

Finally, let us consider the quantity $\Delta_n(z)^{-1} \Delta_n(w) = \widehat{\Delta_n}(z)^{\mathsf{T}} \Delta_n(w)$. By using our formulas for $\Delta_n$ and $\widehat{\Delta_n}$, we have
$$\Delta_n(z)^{-1} \Delta_n(w) = (z-w) \big(A^{(1)}_{n+1}\big)_{21} \mathsf{E}^{21} + \mathsf{C} $$
where $\mathsf{C}$ is some constant matrix. However by taking $z = w$ we see that $\Delta_n(z)^{-1} \Delta_n(z) = \mathbb{I}$, and so $\mathsf{C} = \mathbb{I}$.
\end{proof}
Finally we arrive at one of our main theorems.
\begin{theorem}[Christoffel--Darboux-type formula]\label{CDtheorem}
For all $N \geq 1$ we have
$$\mathbb{K}_N(x,y) = -\frac{1}{2\pi \mathrm{i}} \frac{\big(A^{-1}_N(x) A_N(y)\big)_{21}}{x-y}\, ,$$
where $\mathbb{K}_N$ is the projection kernel of \eqref{kernelformula}.
\end{theorem}
\begin{proof}
Since $\lambda n D \gamma \Psi_n = P_n - \lambda P_n^{\prime\prime} + \lambda V^\prime P_n^\prime$ for $n \geq 1$ by \eqref{PsiP}, the kernel of \eqref{kernelformula} may be rewritten as
\begin{align}\label{kernelPsi}
\mathbb{K}_N(x,y) = \sum_{n=1}^{N-1} \frac{\lambda n D \gamma}{h_n} P_n(x) \Psi_n(y) + \frac{1}{h_0} \, .
\end{align}
Define $$B_n(x,y) \overset{\mathrm{def}}{=} A_n^{-1}(x) A_n(y) = \widehat{A_n}(x)^\mathsf{T} A_n(y) \, .$$ By \cref{recurrence} we have
\begin{align*}
B_{n+1}(x,y) &= A_n^{-1}(x) \Delta_n(x)^{-1}\Delta_n(y) A_n(y) = A_n^{-1}(x) \Big\{ (x-y) \big(A^{(1)}_{n+1}\big)_{21} \mathsf{E}^{21} + \mathbb{I} \Big \} A_n(y) \\
&=  \big(A^{(1)}_{n+1}\big)_{21} (x-y) A_n^{-1}(x)\mathsf{E}^{21} A_n(y) + B_n(x,y) \, .
\end{align*}
Taking the $(2,1)$ matrix element of both sides we find
\begin{align}\label{telescope}
\big(B_{n+1}\big)_{21}(x,y) = \big(A^{(1)}_{n+1}\big)_{21} (x-y) \big(A_n^{-1}\big)_{22}(x) \big(A_n\big)_{11}(y) + \big(B_n\big)_{21}(x,y) \, .
\end{align}
We then observe by \cref{RHP2solution} that $\big(A_n^{-1}\big)_{22} = P_n$, by \cref{RHP1solution} that $\big(A_n\big)_{11} = \Psi_n$, and, from the same theorem, that $\big(A^{(1)}_{n+1}\big)_{21} = - \frac{2\pi \mathrm{i} \lambda D n \gamma}{h_n}$ for $n \geq 1$. Putting these things together, and summing \eqref{telescope} over $n \in [\![ 1, N-1 ]\!]$, we obtain a telescoping series, whence
$$(x-y)\mathbb{K}_N(x,y) = - \frac{1}{2\pi \mathrm{i}}\Big( \big(B_N\big)_{21}(x,y) - \big(B_1\big)_{21}(x,y)\Big)  + (x-y) \frac{1}{h_0}  \, .$$
Finally, it remains to show that $\big(B_1\big)_{21}(x,y) =  - 2\pi \mathrm{i} (x-y) \frac{1}{h_0}$, which would complete the proof. From the definition we have
\begin{align*}
\big(B_1\big)_{21}(x,y) = \sum_{k=1}^D ( \widehat{A_1})_{k2}(x) (A_1)_{k1}(y) \, .
\end{align*}
From the fact that $B_1(x,x) = \mathbb{I}$, $\big(B_1\big)_{21}(x,x) = 0$. Hence we may write
\begin{align*}
\big(B_1\big)_{21}(x,y) = \sum_{k=1}^D ( \widehat{A_1})_{k2}(x) \Big[  (A_1)_{k1}(y) - (A_1)_{k1}(x) \Big]  \, .
\end{align*}
Next, observe that $(A_1)_{21}$, from \cref{RHP1solution}, is a constant, and so the $k = 2$ term in the above sum gives $0$. Next, if we return to the second column of RHP \ref{RHP2}, we have already shown that the $(2,2)$ element is a monic polynomial of degree $1$, and the remaining elements of the second column are constants. Hence we may write $\big(B_1\big)_{21}(x,y) = p(y) - p(x)$ where
$$p(x) = \sum_{\substack{k=1 \\ k \neq 2}}^D (\widehat{A_1})_{k2} (A_1)_{k1}(x) $$
where we remind that $(\widehat{A_1})_{k2}$ are constants for $k \neq 2$. Let us add by hand the $k = 2$ term evaluated at $x$.
\begin{align*}
p(x) = \underbrace{\sum_{k=1}^D (\widehat{A_1})_{k2}(x) (A_1)_{k1}(x)}_{=(A_1^{-1}(x) A_1(x))_{21} = 0} - (\widehat{A_1})_{22}(x) (A_1)_{21}(x) \, .
\end{align*}
We already know from \cref{RHP1solution} that $(A_1)_{21}(x) = - \frac{2\pi\mathrm{i}}{h_0}$ and $(\widehat{A_1})_{22}(x) = P_1(x) =  x + c$ for some constant $c$. Hence
$$\big(B_1\big)_{21}(x,y)  = p(y) - p(x) = - \frac{2\pi \mathrm{i}}{h_0}(x-y)$$
which completes the proof.
\end{proof}
We should remark that our method of proof for \cref{CDtheorem} is strongly inspired by method of Bertola and Bothner in \cite{Bertola:2015aa} (see in particular, the proof of Theorem 2.8).

\section{Discussion}\label{sec:discussion}

Let us end by discussing the significance of our results. Firstly, we note that we have succeeded in semi-explicitly characterising the image of the map $P \mapsto P - \lambda P^{\prime\prime} + \lambda V^\prime P^\prime$ in terms of linear functionals which are completely independent of the degree of the polynomials. These linear functionals are expressed in terms of the data of a spectral problem \eqref{schroedinger}, namely the even Sibuya solutions $\{ Y_{2j} \}_{j = 0}^{D-1}$ and the constants $\{\beta_j, \kappa_j \}_{j \in [\![ 1, D-1 ]\!] \setminus \{ D/2\}}$, and once $V$ and $\lambda > 0$ are given, these data are fully determined. The fact that the resulting Riemann--Hilbert problem gives exactly the projection kernel we want \eqref{CDformula} suggests we are on the right track. To the knowledge of the author, this represents the first time that either a Riemann--Hilbert problem or a Christoffel--Darboux formula has been found for any class of Sobolev orthogonal polynomials with a continuous inner product.

Perhaps the principal motivation for developing a Riemann--Hilbert representation is the possibility of doing a Deift--Zhou steepest descent analysis \cite{deift_zhou_1993}. For this reason, the fact that the jump matrices in RHPs \ref{RHP1} and \ref{RHP2} are not explicit might be regarded as frustrating. However the situation is not entirely hopeless. Firstly, the $\lambda \to 0$ and $\lambda \to +\infty$ regimes could be analysed with WKB methods to get approximations on the spectral data; indeed, the  formula \eqref{integralequationsolution} already gives an asymptotic series for the Sibuya solutions in the regime of $\lambda \to +\infty$.

Perhaps the most interesting regime is where both $V$ and $\lambda$ scale with $N$ such that $V = N v$ and $\lambda = \frac{1}{N^2}\mu$ for some fixed potential $v$ and some parameter $\mu > 0$. Then the spectral problem \eqref{schroedinger} becomes
\begin{align}\label{exactWKB}
Y_{k}^{\prime\prime}(z) = N^2 \Big\{  \frac{1}{4} v^{\prime}(z)^2   + \mu^{-1} - \frac{1}{2N} v^{\prime\prime}(z) \Big\} Y_k(z) \, .
\end{align}
The problem of finding the asymptotics of Sibuya solutions $Y_k$ for ODEs of the form \eqref{exactWKB} as $N \to +\infty$ is  precisely what \textit{exact WKB theory} describes \cite{iwaki2026leshoucheslecturesexact,kawai_takei_2005}. We speculate that away from the turning points of $\frac{1}{4} (v^{\prime})^2   + \mu^{-1} - \frac{1}{2N} v^{\prime\prime}$ one would use the exact WKB asymptotics and, from there, one would build an outer parametrix for the Riemann--Hilbert problem $A_n$, whilst in a neighbourhood of the turning points one would be required to do a local analysis. We leave such fascinating yet highly technical problems for future research. Let us also briefly mention the celebrated ODE/IM correspondence \cite{Dorey_2007} which concerns the Sibuya solutions $Y_k$ and Stokes multipliers $\sigma_k$ in the case that $U_0$ is a monomial. This special form of the potential introduces additional symmetries into the problem which greatly simplifies things, and allows a bridge to be set up with the world of quantum integrable models. We mention this because although there is no \textit{exact} connection, since no polynomial $V$ can make $U_0$ a monomial, there may be an \textit{asymptotic} connection, since if $v$ is monomial the leading term in \eqref{exactWKB} is of the form of a monomial plus a constant.

The method of the present paper relies strongly on the fact that, firstly, there are only two derivatives in total in our Sobolev inner product \eqref{sobolevinnerprod}, and secondly, that the \textit{same} polynomial $V$ appears in the weight functions of both terms. It is not at all clear to what extent our construction generalises beyond the class of inner products considered in \eqref{sobolevinnerprod}, and we leave this as a problem for future investigation.

Finally, let us remark on the non-Hermitian case, i.e. when the inner product \eqref{sobolevinnerprod} takes the same form but the contour $\mathbb{R}$ is replaced by some other unbounded contour. That is, one should pick a pair of distinct integers $a, b \in [\![0, D-1]\!]$ so that the contour goes from $\infty_{2a}$ to $\infty_{2b}$. In this case $D$ is no longer required to be even, and indeed we could have $D = 3$, and $V$ could have complex coefficients. This immediately has the problem that the non-Hermitian 
\enquote{inner product} $\langle \cdot, \cdot \rangle_S$ might be degenerate, hence non-degeneracy must be added as an additional hypothesis. Furthermore, one would need to find a suitable generalisation of \cref{Mlemma} to the non-Hermitian setting, since the proof relied essentially on $D$ being even and $V$ being real.

\section*{Acknowledgement}

A.L. is supported by the joint ANR-DFG TSF24 project ANR-24-CE92-0033.

\bibliographystyle{plain}
\bibliography{SobolevRefs} 

\end{document}